\documentclass[10pt]{article}
\usepackage{amsthm}
\usepackage{amsmath}
\usepackage{amssymb}
\usepackage{makeidx}

\theoremstyle{definition}
\newtheorem{definition}{Definition}[section]
\newtheorem{notation}[definition]{Notation}
\newtheorem{example}[definition]{Example}

\newtheorem{remark}[definition]{Remark}
\newtheorem{note}[definition]{Note}

\theoremstyle{plain}
\newtheorem{theorem}[definition]{Theorem}
\newtheorem{lemma}[definition]{Lemma}
\newtheorem{proposition}[definition]{Proposition}
\newtheorem{corollary}[definition]{Corollary}

\newcommand{\beq}{\begin{equation}}

\newcommand{\eeq}{\end{equation}}

\newcommand{\bdfn}{\begin{definition}}

\newcommand{\edfn}{\end{definition}}

\newcommand{\bthm}{\begin{theorem}}

\newcommand{\ethm}{\end{theorem}}

\newcommand{\bprop}{\begin{proposition}}

\newcommand{\eprop}{\end{proposition}}

\newcommand{\bcor}{\begin{corollary}}

\newcommand{\ecor}{\end{corollary}}

\newcommand{\blem}{\begin{lemma}}

\newcommand{\elem}{\end{lemma}}

\newcommand{\bex}{\begin{example}}

\newcommand{\eex}{\end{example}}

\newcommand{\bxc}{\begin{exercise}}

\newcommand{\exc}{\end{exercise}}

\newcommand{\bntn}{\begin{notation}}

\newcommand{\entn}{\end{notation}}

\newcommand{\be}{\begin{enumerate}}

\newcommand{\ee}{\end{enumerate}}

\newcommand{\bce}{\begin{center}}

\newcommand{\ece}{\end{center}}

\newcommand{\bi}{\begin{itemize}}

\newcommand{\ei}{\end{itemize}}

\newcommand{\bt}{\begin{tabular}}

\newcommand{\et}{\end{tabular}}

\newcommand{\ba}{\begin{array}} 

\newcommand{\ea}{\end{array}}

\numberwithin{equation}{section}

\def\N{{\mathbb N}}

\newcommand {\bua} {\begin{eqnarray*}}

\newcommand {\eua} {\end {eqnarray*}}

\begin{document}
\title{Algebraic and Topological Results on\\ Lifting Properties in Residuated Lattices}
\author{Daniela CHEPTEA, George GEORGESCU and Claudia MURE\c SAN\\ \footnotesize University of Bucharest\\ \footnotesize Faculty of Mathematics and Computer Science\\ \footnotesize Academiei 14, RO 010014, Bucharest, Romania\\ \footnotesize Emails: d.cheptea@gmail.com; georgescu.capreni@yahoo.com; c.muresan@yahoo.com, cmuresan@fmi.unibuc.ro}
\date{\today }
\maketitle

\begin{abstract} We define lifting properties for universal algebras, which we study in this general context and then particularize to various such properties in certain classes of algebras. Next we focus on residuated lattices, in which we investigate lifting properties for Boolean and idempotent elements modulo arbitrary, as well as specific kinds of filters. We give topological characterizations to the lifting property for Boolean elements and several properties related to it, many of which we obtain by means of the reticulation.\\ {\em 2010 Mathematics Subject Classification:} Primary: 06F35; secondary: 03G25, 03C05, 08A30.\\ {\em Key words and phrases:} universal algebra; congruence; (local, semilocal, maximal, hyperarchimedean, simple, semisimple) residuated lattice; filter of a (bounded distributive, residuated) lattice; (Boolean, Idempotent) Lifting Property; Boolean center; idempotent element; regular element; nilpotent element; (prime, maximal) spectrum; Stone topology; radical; reticulation; functor.\end{abstract}

\section{Introduction}

Idempotent Lifting Property (abbreviated ILP, or LIP in some works) has occurred in ring theory in relation to clean rings and exchange rings (\cite{lam}, \cite{nic}). A ring has the ILP iff its idempotents can be lifted modulo every left ideal (\cite{nic}). Rings with ILP have been given many characterizations, some of algebraic nature, some topological (\cite[Theorem $1.7$]{mcgov}). Commutative rings with ILP coincide to commutative clean rings and to commutative exchange rings (\cite{nic}).

A lifting property modulo the radical for Boolean elements in MV--algebras has been used in \cite{figele} for characterizing maximal MV--algebras, then generalized to BL--algebras in \cite{leo} and residuated lattices in \cite{eu3}, \cite{eu}. Residuated lattices with the Boolean Lifting Property (BLP) modulo the radical have been introduced and studied from the algebraic point of view in \cite{eu3}. The BLP modulo the radical in a residuated lattice $A$ means that all Boolean elements of the quotient residuated lattice $A/{\rm Rad}(A)$ are classes of Boolean elements of $A$ modulo ${\rm Rad}(A)$. Maximal residuated lattices with BLP modulo the radical have turned out to satisfy strong representation theorems.

In \cite{ggcm}, we have generalized the BLP for residuated lattices to all filters instead of just the radical, we have identified several classes of residuated lattices with BLP, such as Boolean algebras, chains, local and hyperarchimedean residuated lattices, and we have obtained several characterizations of the BLP, as well as certain structure theorems for residuated lattices with BLP.

In the present paper, we continue the study we have begun in \cite{ggcm}, obtaining further algebraic properties, as well as topological characterizations for residuated lattices with BLP, and we also study lifting properties for other types of elements, all of which were inspired by the ILP in the case of rings. We start by introducing lifting properties in the general context of universal algebras, from which all particular lifting properties naturally occur. We exemplify by certain such properties for unitary commutative rings, for bounded distributive lattices, and then we restrict out research to residuated lattices, where, along with BLP, we have lifting properties for idempotent elements (ILP) and for regular elements (RLP). It turns out that the RLP is trivial, but the BLP and the ILP are not, nor do they coincide, except in some remarkable particular cases. BLP turns out to relate to important topological properties, for the study of which the reticulation of a residuated lattice proves very useful. The reticulation for residuated lattices, that has been defined in \cite{eu1}, \cite{eu} and studied in \cite{eu1}, \cite{eu}, \cite{eu2}, \cite{eu4}, \cite{eu5}, \cite{eu7}, is essentially the construction of a functor from the category of residuated lattices to the category of bounded distributive lattices which takes residuated lattices to bounded distributive lattices with the same topological structure for the prime spectrum. As announced and exemplified in \cite{eu2}, \cite{eu5}, this functor permitts a very easy and fruitful transfer of properties between the two categories in question.

Section \ref{preliminaries} of our paper is a brief introduction to the theory of residuated lattices, in which we collect previously known results that we use in the sequel. The following sections consist of new and original results belonging to the authors of the present paper, and very few previously known results, which we mention when they occur and out of which we provide some with new proofs.

Section \ref{lpsunivalg} contains a generic definition for lifting properties in universal algebras, which we study in this general context. We then exemplify how this general theory can be applied to particular classes of algebras and specific kinds of lifting properties. We feel that these examples illustrate the potential of this unified theory for the study of lifting properties.

Section \ref{lpsreslat} is concerned with an algebraic study of the BLP and ILP for the class of residuated lattices and some of its subclasses. These two types of lifting properties are studied individually, as well as in relation to each other. While, for instance, in MV--algebras the two lifting properties coincide, we provide an example that shows that this is not the case in every residuated lattice, and, while we evidentiate certain classes of residuated lattices which have BLP or ILP or both of these lifting properties, we also give an example which proves that these classes do not cover all residuated lattices satisfying these properties.

Section \ref{specretic} begins by a recollection of a series of well known results on the prime and maximal spectrum of a residuated lattice, followed by a brief description of the reticulation of a residuated lattice, in which we provide the definition of the reticulation and summarize those of its properties that we use in the sequel. We then prove that the reticulation functor both preserves and reflects the BLP, and provide a first illustration of the usefulness of this functorial property.

The first part of Section \ref{topchar} is concerned with residuated lattices with the property that any of their prime filters is included in a unique maximal filter. This is similar to the property that defines Gelfand rings (\cite{joh}) and conormal lattices (\cite{whcor}). We call the residuated lattices with this property Gelfand residuated lattices, and we obtain many characterizations for them by transferring analogous results on conormal lattices through the reticulation functor. We also obtain a series of topological characterizations of residuated lattices with BLP and some of their subclasses, by using the fact that all residuated lattices with BLP are Gelfand.

\section{Preliminaries}
\label{preliminaries}

Throughout this paper, every algebraic structure will be designated by its support set, whenever this is useful and there is no danger of confusion.

The set of the natural numbers will be denoted by $\N $, and we shall use the notation $\N ^*$ for the set of the nonzero natural numbers.
 
In this section we recall several known notions and results that we use in the rest of the paper. For a further study of these notions and results, as well as others that the reader may need to review, one may consult \cite{bal}, \cite{gal}, \cite{haj}, \cite{ior}, \cite{kow}, \cite{pic}, \cite{tur}. 

\begin{definition} A {\em commutative integral bounded residuated lattice} (in brief, a {\em residuated lattice}) is an algebra $(A,\vee ,\wedge ,\odot ,\rightarrow ,0,1)$, where $\vee ,\wedge ,\odot ,\rightarrow $ are binary operations on $A$ (called {\em join, meet, multiplication} and {\em implication} or {\em residuum}, respectively) and $0,1\in A$, such that $(A,\vee ,\wedge ,0,1)$ is a bounded lattice (whose partial order will be denoted by $\leq $), $(A,\odot ,1)$ is a commutative monoid and the following equivalence, called {\em the law of residuation}, holds for all $a,b,c\in A$: $a\leq b\rightarrow c$ iff $a\odot b\leq c$.\end{definition}

In any residuated lattice $A$, the following derivative operations are defined: for all $a,b\in A$, $a\leftrightarrow b=(a\rightarrow b)\wedge (b\rightarrow a)$ (the {\em equivalence} or the {\em biresiduum}) and $\neg \, a=a\rightarrow 0$ (the {\em negation}). We shall also use the alternate notation for the equivalence: $d(a,b)=a\leftrightarrow b$, for all $a,b\in A$. Also, for all $a\in A$ and all $n\in \N ^{*}$, we shall denote by $a^n=\underbrace{\textstyle a\odot \ldots \odot a}_{\textstyle n\ {\rm of}\ a}$ and by $a^{0}=1$.

Residuated lattices are non--empty, because they contain the constants $0$ and $1$. The one--element residuated lattice (that is the residuated lattice with $0=1$) is called {\em the trivial residuated lattice.} Any residuated lattice with at least two elements (that is with $0\neq 1$) is said to be {\em non--trivial.}

Morphisms of residuated lattices are called, in brief, {\em residuated lattice morphisms}.

If $A$ is a residuated lattice, then $\odot $ is distributive with respect to $\vee $ (see Lemma \ref{latrez}, (\ref{latrez9}), below), while the bounded lattice $(A,\vee ,\wedge ,0,1)$ is not necessarily distributive, but it is uniquely complemented (see Lemma \ref{cbool}, (\ref{cbool0}), below). The residuated lattice $A$ is said to be {\em distributive} iff its underlying bounded lattice, $(A,\vee ,\wedge ,0,1)$, is distributive.

Two classes of examples of distributive residuated lattices are chains (see \cite{ggcm} for how they can be organized as residuated lattices) and Boolean algebras, which, as pointed out in \cite{ggcm}, can be organized as residuated lattices in only one way, namely: if ${\cal A}=(A,\vee ,\wedge ,\bar{\ },0,1)$ is a Boolean algebra, then the only residuated lattice whose underlying bounded lattice is ${\cal A}$ is $(A,\vee ,\wedge ,\odot ,\rightarrow ,0,1)$, where $\odot =\wedge $ and $\rightarrow $ is the Boolean implication: for all $a,b\in A$, $a\rightarrow b=\overline{a}\vee b$; in such a residuated lattice, $\neg \, a=\overline{a}$ for all $a\in A$, as shown by Lemma \ref{cbool}, (\ref{cbool0}), below.

We shall call any linearly orderred residuated lattice a {\em residuated chain.}

We make the usual convention concerning the priority assigned to the operations defined above and all other operations we shall be using in what follows: constants and variables have the highest priority, exponentiation is next, negation follows, and the lowest priority goes to binary operations of any kind.

Throughout the rest of this section, $A$ will be an arbitrary residuated lattice and $L$ will be an arbitrary bounded distributive lattice.

Although, for instance, the commutativity of certain binary operations makes some of the properties below redundant, we have chosen to write these properties, for the sake of clarity.

\begin{lemma} For all $a,b,x,y\in A$, we have:
\begin{enumerate}
\item\label{latrez9} $a\odot (x\vee y)=(a\odot x)\vee (a\odot y)$;
\item\label{latrez8} $a\odot b\leq a$ and $a\odot b\leq b$; consequently, $a\odot b\leq a\wedge b$ and $a\odot 0=0$;
\item\label{latrez6} if $a\leq x$ and $b\leq y$, then $a\odot b\leq x\odot y$; consequently, for all $n\in \N ^*$, $a^n\leq a$, and, moreover, for all $k,n\in \N ^*$ such that $k\leq n$, we have $a^n\leq a^k$;
\item\label{latrez10} if $a\vee b=1$, then $a\wedge b=a\odot b$;
\item\label{latrez5} $a\leq b$ iff $a\rightarrow b=1$; $1\rightarrow a=a$; consequently, $\neg \, 0=1$ and $\neg \, 1=0$;
\item\label{latrez7} $a\leq \neg \, \neg \, a$; $\neg \, \neg \, \neg \, a=\neg \, a$; if $a\leq b$, then $\neg \, b\leq \neg \, a$;
\item\label{latrez1} $a\odot \neg \, a=0$; $a\odot b=0$ iff $a\leq \neg \, b$ iff $b\leq \neg \, a$.\end{enumerate}\label{latrez}\end{lemma}

We shall denote by ${\cal I}(A)$ the set of the {\em idempotent elements} of $A$ and by ${\rm Reg}(A)$ the set of the {\em regular elements} of $A$, that is:

\begin{itemize}
\item ${\cal I}(A)=\{e\in A\ |\ e^2=e\}=\{e\in A\ |\ (\forall \, n\in \N ^*)\, (e^n=e)\}$, where the first equality is the definition and the second equality is trivial; $A$ is called a {\em $G\ddot{o}del$ algebra} iff $\odot =\wedge $ in $A$; according to \cite[Proposition $3.1$]{eu2}, ${\cal I}(A)=A$ iff $A$ is a ${\rm G\ddot{o}del}$ algebra; for instance, the residuated lattices induced by Boolean algebras are ${\rm G\ddot{o}del}$ algebras;
\item ${\rm Reg}(A)=\{e\in A\ |\ \neg \, \neg \, e=e\}=\{\neg \, e\ |\ e\in A\}$, where the first equality is the definition and the second equality is straightforward from Proposition \ref{latrez}, (\ref{latrez7}); $A$ is said to be {\em involutive} iff ${\rm Reg}(A)=A$.\end{itemize}

An element $a\in A$ is said to be {\em nilpotent} iff $a^n=0$ for some $n\in \N ^*$.

See the works cited at the beginning of this section for the following particular kinds of residuated lattices: MV--algebras and BL--algebras; the first form a subclass of the latter; more precisely, MV--algebras are exactly the involutive BL--algebras.

The set of the complemented elements of $L$ is denoted by ${\cal B}(L)$ and called {\em the Boolean center of $L$}. It is well known and immediate that ${\cal B}(L)$, with the bounded lattice operations induced by those of $L$, together with the complementation operation, is a Boolean algebra.

The set of the complemented elements of the bounded lattice $(A,\vee ,\wedge ,0,1)$ is denoted by ${\cal B}(A)$ and called {\em the Boolean center of the residuated lattice $A$}. $({\cal B}(A),\vee ,\wedge ,\neg \, ,0,1)$ is a Boolean algebra, with the operations induced by those of $A$, where the complementation is given by the negation in $A$ (see Lemma \ref{cbool}, (\ref{cbool0}), below).

Obviously, ${\cal B}({\cal A})=A$ iff ${\cal A}$ is induced by a Boolean algebra as described above.

\begin{lemma}
For any $a\in A$ and any $e,f\in {\cal B}(A)$, we have:

\begin{enumerate}

\item\label{cbool1} $a\in {\cal B}(A)$ iff $a\vee \neg \, a=1$;
\item\label{cbool0} the complement of $e$ in the Boolean algebra ${\cal B}(A)$ is $\neg \, e$; consequently, ${\cal B}(A)\subseteq {\rm Reg}(A)$;
\item\label{cbool3} $e\odot f=e\wedge f$; consequently, $e^2=e$ and, therefore, ${\cal B}(A)\subseteq {\cal I}(A)$;
\item\label{cbool2} $e\rightarrow a=\neg \, e\vee a$.
\end{enumerate}\label{cbool}\end{lemma}

\begin{remark}\begin{enumerate}\item\label{imme1} By Lemma \ref{cbool}, (\ref{cbool0}) and (\ref{cbool3}), $\{0,1\}\subseteq {\cal B}(A)\subseteq {\cal I}(A)\cap {\rm Reg}(A)$.
\item\label{imme2} But, as shown by Examples \ref{exlpdif} and \ref{nice}, not only do the converse inclusions in Lemma \ref{cbool}, (\ref{cbool0}) and (\ref{cbool3}), not always hold, but neither does the converse of the inclusion ${\cal B}(A)\subseteq {\cal I}(A)\cap {\rm Reg}(A)$, and, moreover, no inclusion between ${\cal I}(A)$ and ${\rm Reg}(A)$ holds in every residuated lattice $A$.\end{enumerate}\label{imme}\end{remark}

Lemma \ref{cbool}, (\ref{cbool3}), shows that, for every $e\in {\cal B}(A)$, $[e)=\{a\in A\ | e\leq a\}$.

By defining, for any residuated lattices $R$ and $S$ and any residuated lattice morphism $\varphi :R\rightarrow S$, ${\cal B}(\varphi ):{\cal B}(R)\rightarrow {\cal B}(S)$, for all $e\in {\cal B}(R)$, ${\cal B}(\varphi )(e)=\varphi (e)$, we get a well--defined Boolean morphism ${\cal B}(\varphi )$, and thus a covariant functor ${\cal B}$ from the category of residuated lattices to the category of Boolean algebras.

Similarly, if we define, for any bounded distributive lattices $M$ and $P$ and any bounded lattice morphism $h:M\rightarrow P$, ${\cal B}(h):{\cal B}(M)\rightarrow {\cal B}(P)$, for all $e\in {\cal B}(M)$, ${\cal B}(h)(e)=h(e)$, then we get a well--defined Boolean morphism ${\cal B}(h)$, and thus a covariant functor ${\cal B}$ from the category of bounded distributive lattices to the category of Boolean algebras.

We consider that the coincidence of notations between the two functors defined above poses no danger of confusion.

An element $a\in A$ is said to be {\em archimedean} iff $a^n\in {\cal B}(A)$ for some $n\in \N ^*$. $A$ is said to be {\em hyperarchimedean} iff all of its elements are archimedean. As examples of hyperarchimedean residuated lattices, we have Boolean algebras, whose underlying set equals their Boolean center.

A {\em filter of $A$} is a non--empty subset $F$ of $A$ such that, for all $a,b\in A$:

\begin{itemize}
\item if $a,b\in F$, then $a\odot b\in F$;
\item if $a\in F$ and $a\leq b$, then $b\in F$.\end{itemize}

The set of filters of $A$ will be denoted by ${\rm Filt}(A)$. $\{1\}\in {\rm Filt}(A)$ and $A\in {\rm Filt}(A)$; $\{1\}$ is called {\em the trivial filter} and $A$ is called {\em the improper filter of $A$.} Any $F\in {\rm Filt}(A)$ such that $F\neq A$ is called a {\em proper filter of $A$.} Obviously, $1$ belongs to any filter of $A$, while a filter $F$ of $A$ is proper iff $0\notin F$.

Notice that, for any filter $F$ of $A$ and any $a,b\in A$, we have the following equivalences: $a,b\in F$ iff $a\odot b\in F$ iff $a\wedge b\in F$ (see Proposition \ref{latrez}, (\ref{latrez8})).

Clearly, an arbitrary intersection of filters of $A$ is a filter of $A$. For any $X\subseteq A$, we shall denote by $[X)$ the smallest filter of $A$ which includes $X$, that is $\displaystyle [X)=\bigcap _{\stackrel{\scriptstyle F\in {\rm Filt}(A)}{\scriptstyle F\supseteq X}}F$; $[X)$ is called {\em the filter of $A$ generated by $X$}. For any $x\in A$, $[\{x\})$ is denoted, simply, by $[x)$, and it is called {\em the principal filter of $A$ generated by $x$}.

For all $x\in A$, $[x)=\{a\in A\ |\ (\exists \, n\in \N ^*)\, (x^n\leq a)\}$. Clearly, if $x$ is an idempotent element of $A$, then $[x)=\{a\in A\ |\ x\leq a\}$. Consequently, by Lemma \ref{cbool}, (\ref{cbool3}), if $x\in {\cal B}(A)$, then $[x)=\{a\in A\ |\ x\leq a\}$. The set of the principal filters of $A$ will be denoted by ${\rm PFilt}(A)$.

It is immediate that, for all $a,b\in A$:

\begin{itemize}
\item for all $n\in \N ^*$, $[a^n)=[a)$; 
\item $[a)=\{1\}$ iff $a=1$; $[a)=A$ iff $a$ is nilpotent;
\item $[a)\vee [b)=[a\odot b)=[a\wedge b)$; $[a)\cap [b)=[a\vee b)$.\end{itemize}

It is easy to notice that, for any $X\subseteq A$, $[X)=\{a\in A\ |\ (\exists \, n\in \N ^*)\, (\exists \, x_1,\ldots ,x_n\in X)\, (x_1\odot \ldots \odot x_n\leq a)\}$. From this it follows that any finitely generated filter of a residuated lattice is principal, because, clearly, $[\{x_1,\ldots ,x_n\})=[x_1\odot \ldots \odot x_n)$, for any $n\in \N ^*$ and any $x_1,\ldots ,x_n\in A$. Thus, in particular, any finite filter of a residuated lattice is principal, since, trivially, $F=[F)$ for all $F\in {\rm Filt}(A)$.

For any $F,G\in {\rm Filt}(A)$, we denote $F\vee G=[F\cup G)$. More generally, for any family $(F_i)_{i\in I}\subseteq {\rm Filt}(A)$, we shall denote $\displaystyle \bigvee _{i\in I}F_i=[\bigcup _{i\in I}F_i)$. Then $({\rm Filt}(A),\vee ,\cap ,\{1\},A)$ becomes a bounded distributive lattice (and clearly a complete one), orderred by set inclusion. Also, ${\rm PFilt}(A)$ is a bounded sublattice of ${\rm Filt}(A)$.

For any $F,G\in {\rm Filt}(A)$, $F\vee G=\{a\in A\ |\ (\exists \, x\in F)\, (\exists \, y\in G)\, (x\odot y\leq a)\}$.

The maximal elements of $({\rm Filt}(A)\setminus \{A\},\subseteq )$ are called {\em maximal filters of $A$.} The set of the maximal filters of $A$ is denoted by ${\rm Max}(A)$ and called {\em the maximal spectrum of $A$.}

A proper filter $P$ of $A$ with the property that, for all $a,b\in A$, if $a\vee b\in P$ then $a\in P$ or $b\in P$, is called a {\em prime filter of $A$.} The set of the prime filters of $A$ is denoted by ${\rm Spec}(A)$ and called {\em the (prime) spectrum of $A$.}

The following inclusion holds: ${\rm Max }(A)\subseteq {\rm Spec}(A)$.

Any proper filter of $A$ is included in a maximal (and thus a prime) filter of $A$. Moreover, any filter of $A$ equals the intersection of the prime filters that include it. Consequently, the intersection of all the prime filters of $A$ is $\{1\}$. Any nonzero element of $A$ is contained in a maximal (and thus a prime) filter of $A$.

\begin{lemma} If $M$ is a filter of $A$, then the following are equivalent:

\begin{itemize}
\item $M$ is a maximal filter of $A$;
\item any $a\in A$ satisfies: $a\notin M$ iff there exists an $n\in \N ^*$ such that $\neg \, a^n\in M$.\end{itemize}\label{caractfiltmax}\end{lemma}

\begin{remark} If $A$ is a BL--algebra, then every prime filter of $A$ is included in a unique maximal filter. In particular, this holds when $A$ is an MV--algebra.\label{blgelfand}\end{remark} 

The intersection of all maximal filters of $A$ is called the {\em radical of $A$} and it is denoted by ${\rm Rad}(A)$. By the above, ${\rm Rad}(A)$ is a filter of $A$. $A$ is said to be {\em semisimple} iff ${\rm Rad}(A)=\{1\}$.

If $F$ is a filter of $A$, then the binary relation $\equiv ({\rm mod}\ F)$ on $A$, defined by: for all $x,y\in A$, $x\equiv y({\rm mod}\ F)$ iff $x\leftrightarrow y\in F$, is a congruence of the residuated lattice $A$, called {\em the congruence modulo $F$}. For every $x\in A$, we shall denote by $x/F$ the congruence class of $x$ with respect to $\equiv ({\rm mod}\ F)$. Residuated lattices form an equational class, hence the quotient set $A/F=\{x/F\ |\ x\in A\}$ with respect to the congruence $\equiv ({\rm mod}\ F)$ becomes a residuated lattice with the operations defined canonically. Clearly, for all $x\in F$, $x/F=1/F=F$.

It is easy to prove that, for any filter $F\subseteq {\rm Rad}(A)$, ${\rm Rad}(A/F)={\rm Rad}(A)/F$. Consequently, $A/{\rm Rad}(A)$ is semisimple.

The residuated lattice $A$ is said to be {\em local} iff it has exactly one maximal filter. $A$ is said to be {\em semilocal} iff it has only a finite number of maximal filters. $A$ is said to be {\em maximal} iff, for any non--empty index set $I$, any $(a_i)_{i\in I}\subseteq A$ and any $(F_i)_{i\in I}\subseteq {\rm Filt}(A)$ such that $\displaystyle \bigcap _{j\in J}a_j/F_j$ is non--empty for every finite subset $J$ of $I$, it follows that $\displaystyle \bigcap _{i\in I}a_i/F_i$ is non--empty. Maximal residuated lattices are semilocal. 

All the considerations above related to filters hold if we replace the arbitrary residuated lattice $A$ by the arbitrary bounded distributive lattice $L$, with only these modifications: replacing $\odot $ by $\wedge $ and defining the congruence modulo a filter $F$ of $L$ by: for any $x,y\in L$, $x\equiv y\, ({\rm mod}\ F)$ iff there exists $a\in F$ such that $x\wedge a=y\wedge a$. We shall be using the same notations as the ones above concerning filters of a residuated lattice for an arbitrary bounded distributive lattice instead of the arbitrary residuated lattice, and we believe that there is no danger of this producing a confusion.

In the case of residuated lattices, the correspondence above is a bijection between the set of filters and the set of congruences; actually, it is even a bounded lattice isomorphism between the bounded lattice of filters and that of congruences; more on this subject in the next section of the paper.

The fact that residuated lattices form an equational class also ensures us that arbitrary direct products of residuated lattices become residuated lattices with the operations defined canonically. 

As shown in \cite[Proposition $2.16$]{eu3}, for any $e\in {\cal B}(A)$, $([e),\vee ,\wedge ,\odot ,\rightarrow _e,e,1)$ is a residuated lattice, where, for all $a,b\in [e)$, $a\rightarrow _eb=e\vee (a\rightarrow b)$, and all other binary operations are induced by those of $A$.

\section{Lifting Properties in Universal Algebras}
\label{lpsunivalg}

Idempotent Lifting Property (ILP) is studied in ring theory (\cite{andcam}, \cite{camyu}, \cite{hannic}, \cite{nic}), while Boolean Lifting Property (BLP) appears in MV--algebras (\cite{figele}), BL--algebras (\cite{leo}) and residuated lattices (\cite{eu3}, \cite{ggcm}). In this section we define a general notion of lifting property in the context of universal algebras; this notion of lifting property embodies ILP, BLP and many other important kinds of lifting properties. We then prove some results related to lifting properties in universal algebras and some on lifting properties in particular classes of algebras.

Let $\tau $ be a universal algebras signature and ${\bf L}_{\tau }$ the first order language associated to $\tau $.

Let $\varphi (v)$ be a formula of ${\bf L}_{\tau }$ having at most $v$ as free variable. If ${\cal A}$ is a $\tau $--algebra, then we denote by $A$ its underlying set, by ${\rm Con}({\cal A})$ the set of the congruences of ${\cal A}$ and by ${\cal A}(\varphi )=\{a\in A\ |\ {\cal A}\vDash \varphi (a)\}$. Also, we shall denote by $\Delta _A=\{(a,a)\ |\ a\in A\}\in {\rm Con}({\cal A})$ the binary relation given by the equality on $A$.

For every two congruences $\theta $ and $\omega $ of a $\tau $--algebra ${\cal A}$, we define $\theta \vee \omega $ to be the least congruence of ${\cal A}$ (with respect to $\subseteq $) which includes $\theta \cup \omega $. According to \cite[Theorem $5.3$, p. $40$]{bur}, $({\rm Con}({\cal A}),\vee ,\cap ,\Delta _A,A^2)$ is a complete bounded lattice. If this lattice is distributive, then the $\tau $--algebra ${\cal A}$ is said to be {\em congruence--distributive}.

\begin{definition}
Let ${\cal A}$ be a $\tau $--algebra and $\Omega \subseteq {\rm Con}({\cal A})$.

We say that ${\cal A}$ has the {\em $(\varphi ,\Omega )$--Lifting Property} (in brief, {\em $(\varphi ,\Omega )$--LP}) iff, for every $\theta \in \Omega $ and every $a\in A$, if $a/\theta \in ({\cal A}/\theta )(\varphi )$, then there exists $e\in {\cal A}(\varphi )$ such that $a/\theta =e/\theta $.

The $(\varphi ,{\rm Con}({\cal A}))$--LP will be called, simply, the {\em $\varphi $--Lifting Property} (in brief, {\em $\varphi $--LP}).

If $\theta \in {\rm Con}({\cal A})$, then the $(\varphi ,\{\theta \})$--LP will be denoted, simply, by $(\varphi ,\theta )$--LP. Also, we shall say that $\theta $ has the $\varphi $--LP iff ${\cal A}$ has the $(\varphi ,\theta )$--LP.\end{definition}

\begin{remark}
Clearly, the inclusion ${\cal A}(\varphi )/\theta \subseteq ({\cal A}/\theta )(\varphi )$ holds for every formula $\varphi $ of ${\bf L}_{\tau }$ with at most one free variable, every $\tau $--algebra ${\cal A}$ and every $\theta \in {\rm Con}({\cal A})$.

The definition above says that ${\cal A}$ has the $(\varphi ,\theta )$--LP iff $({\cal A}/\theta )(\varphi )\subseteq {\cal A}(\varphi )/\theta $, which is equivalent to ${\cal A}(\varphi )/\theta =({\cal A}/\theta )(\varphi )$. In other words:

\begin{itemize}
\item $\theta $ has the $\varphi $--LP iff $({\cal A}/\theta )(\varphi )\subseteq {\cal A}(\varphi )/\theta $ iff ${\cal A}(\varphi )/\theta =({\cal A}/\theta )(\varphi )$.
\end{itemize}

Clearly, for every $\Omega \subseteq {\rm Con}({\cal A})$, the $(\varphi ,\Omega )$--LP is equivalent to $(\varphi ,\theta )$--LP for all $\theta \in \Omega $. Consequently, the $\varphi $--LP is equivalent to $(\varphi ,\theta )$--LP for all $\theta \in {\rm Con}({\cal A})$. In other words:

\begin{itemize}
\item ${\cal A}$ has the $(\varphi ,\Omega )$--LP iff every $\theta \in \Omega $ has the $\varphi $--LP;
\item ${\cal A}$ has the $\varphi $--LP iff every $\theta \in {\rm Con}({\cal A})$ has the $\varphi $--LP.
\end{itemize}

Trivially, ${\cal A}$ always has the $(\varphi ,\Delta _A)$--LP and the $(\varphi ,A^2)$--LP.\label{exultriv}\end{remark}

\begin{example}
Let ${\bf L}_{\rm ring}$ be the language of unitary commutative rings. If ${\cal R}$ is an unitary commutative ring, then the lattice ${\rm Con}({\cal R})$ is isomorphic to the lattice ${\rm Id}({\cal R})$ of the ideals of ${\cal R}$.

Let us consider the formula $\varphi (v)$ that expresses the fact that $v$ is an idempotent element of $R$, that is:$$\varphi (v):\ v^2=v.$$Then $\varphi $--LP is the idempotent lifting property in \cite{nic}.

According to \cite{nic}, ${\cal R}$ satisfies the $\varphi $--LP iff ${\cal R}$ is an exchange ring iff ${\cal R}$ is a clean ring.\label{ringex}\end{example}

\begin{example}
Let ${\bf L}_{\rm res}$ be the language of residuated lattices and $\varphi $ be a formula of ${\bf L}_{\rm res}$ having at most one free variable. If ${\cal A}$ is a residuated lattice, then ${\rm Con}({\cal A})$ is isomorphic to the lattice ${\rm Filt}({\cal A})$ of the filters of ${\cal A}$. For every filter $F$ of ${\cal A}$, we shall say that {\em $F$ has the $\varphi $--LP} iff the congruence of ${\cal A}$ associated to the filter $F$ through this bounded lattice isomorphism has the $\varphi $--LP. Thus the residuated lattice ${\cal A}$ has the $\varphi $--LP iff each of its filters has the $\varphi $--LP. Also, recall from Section \ref{preliminaries} that the bounded lattice ${\rm Filt}({\cal A})$ is distributive, which shows that residuated lattices are congruence--distributive algebras (see also \cite{gal}).

The last statement in Remark \ref{exultriv} shows that the trivial filter and the improper filter of ${\cal A}$ always have the $\varphi $--LP.

Now let $\varphi (v)$ be the formula which expresses the first order property that $v$ is a Boolean element of $A$, that is:$$\varphi (v):\ v\vee\neg \, v=1.$$Then ${\cal A}(\varphi )={\cal B}({\cal A})$ and $\varphi $--LP is the Boolean Lifting Property (BLP) in \cite{ggcm}. A filter $F$ of ${\cal A}$ has the BLP iff ${\cal B}({\cal A}/F)={\cal B}({\cal A})/F$. Clearly, if ${\cal B}({\cal A})=A$ (that is if the underlying bounded lattice of ${\cal A}$ is a Boolean algebra), then ${\cal A}$ has BLP.\label{resboolex}\end{example}

\begin{example}
Let ${\cal A}$ be a residuated lattice.

The first order property ``$v$ is an idempotent element in ${\cal A}$`` is formalized by the formula:$$\varphi (v):\ v^2=v.$$Then ${\cal A}(\varphi )={\cal I}({\cal A})$ and $\varphi $--LP is called the Idempotent Lifting Property (ILP). A filter $F$ of ${\cal A}$ has the ILP iff ${\cal I}({\cal A}/F)={\cal I}({\cal A})/F$. Clearly, if ${\cal I}({\cal A})=A$ (which, as mentioned in Section \ref{preliminaries}, is equivalent to the fact that $A$ is a ${\rm G\ddot{o}del}$ algebra), then ${\cal A}$ has ILP.\label{residex}\end{example}

\begin{example}

Let ${\cal A}$ be a residuated lattice.

The first order property ``$v$ is a regular element of ${\cal A}$`` is formalized by the formula:$$\varphi (v):\ v=\neg \, \neg \, v.$$Then ${\cal A}(\varphi )={\rm Reg}({\cal A})$ and $\varphi $--LP is called the Regular Lifting Property (RLP). A filter $F$ of ${\cal A}$ has the RLP iff ${\rm Reg}({\cal A}/F)={\rm Reg}({\cal A})/F$.

RLP is trivial, in the sense that every residuated lattice has RLP. Indeed, for every residuated lattice ${\cal A}$, we have ${\rm Reg}({\cal A})=\{\neg \, a\ |\ a\in A\}$, thus, for every filter $F$ of ${\cal A}$, ${\rm Reg}({\cal A}/F)=\{\neg \, x\ |\ x\in A/F\}=\{\neg \, (a/F)\ |\ a\in A\}=\{\neg \, a/F\ |\ a\in A\}={\rm Reg}({\cal A})/F$.\label{resregex}\end{example}

\begin{example}
Let ${\bf L}_{\rm lat}$ be the first order language of bounded distributive lattices. If ${\cal L}$ is a bounded distributive lattice, then the lattice ${\rm Id}({\cal L})$ of the ideals of ${\cal L}$ and the lattice ${\rm Filt}({\cal L})$ of the filters of ${\cal L}$ are isomorphic to bounded sublattices of ${\rm Con}({\cal L})$. We may consider ${\rm Id}({\cal L})$ and ${\rm Filt}({\cal L})$ to be bounded sublattices of ${\rm Con}({\cal L})$.

The first order property ``$v$ is a Boolean element of ${\cal L}$`` is formalized in ${\bf L}_{\rm lat}$ by the formula:$$\varphi (v):\ \exists \, w\, (v\vee w=1\ \& \ v\wedge w=0).$$Then ${\cal L}(\varphi )={\cal B}({\cal L})$ and, by considerring the bounded sublattices ${\rm Id}({\cal L})$ and ${\rm Filt}({\cal L})$ of ${\rm Con}({\cal L})$, we obtain the properties $(\varphi ,{\rm Id}({\cal L}))$--LP and $(\varphi ,{\rm Filt}({\cal L}))$--LP, respectively.

Here the $(\varphi ,{\rm Con}({\cal L}))$--LP, the $(\varphi ,{\rm Id}({\cal L}))$--LP and the $(\varphi ,{\rm Filt}({\cal L}))$--LP do not necessarily coincide.

Throughout the rest of this paper, unless mentioned otherwise, we shall say that a bounded distributive lattice ${\cal L}$ has the {\em Boolean Lifting Property} (in brief, {\em BLP}) iff ${\cal L}$ satisfies the $(\varphi ,{\rm Filt}({\cal L}))$--LP, that is each filter $F$ of ${\cal L}$ satisfies the {\em BLP}, that is ${\cal B}({\cal L}/F)={\cal B}({\cal L})/F$.

Clearly, if ${\cal B}({\cal L})=L$, that is if ${\cal L}$ is a Boolean algebra, then ${\cal L}$ has BLP.\label{blatex}\end{example}

\begin{remark}
The bounded lattice isomorphism between ${\rm Con}({\cal A})$ and ${\rm Id}({\cal A})$ in any residuated lattice ${\cal A}$ ensures us that, for any formula $\varphi $ of ${\bf L}_{\rm res}$ having at most one free variable, the $\varphi $--LP for ${\cal A}$ (that is the $(\varphi ,{\rm Con}({\cal A}))$--LP) coincides with the $(\varphi ,{\rm Filt}({\cal A}))$--LP.\end{remark}

Let $\varphi $ be an existential formula in ${\bf L}_{\tau }$, that is a formula of the form $\exists \, w_1\, \exists \, w_2\, \ldots \exists \, w_n\, \psi (v,w_1,w_2,\ldots ,w_n)$ is an atomic formula in ${\bf L}_{\tau }$.

\begin{proposition}
Let ${\cal A}$ be a $\tau $--algebra. Then: ${\cal A}$ has the $\varphi $--LP iff ${\cal A}/\theta $ has the $\varphi $--LP for every $\theta \in {\rm Con}({\cal A})$.\label{propcaturi}\end{proposition}

\begin{proof} For the converse implication, just take the congruence $\Delta _A$ of ${\cal A}$. The direct implication follows from the form of ${\rm Con}({\cal A}/\theta )$ and the Second Isomorphism Theorem.\end{proof}

\begin{remark}
\begin{itemize}
\item By applying Proposition \ref{propcaturi} to the language and algebra in Example \ref{ringex}, we get \cite[Proposition $1.4$]{nic}.
\item By applying Proposition \ref{propcaturi} to the language and algebra in Example \ref{resboolex}, we get \cite[Corollary $4.16$]{ggcm}.
\end{itemize}\label{caturiaplic}\end{remark}

Let $n\in \N ^*$, ${\cal A}_1,{\cal A}_2,\ldots ,{\cal A}_n$ be congruence--distributive $\tau $--algebras and $\displaystyle {\cal A}=\prod _{i=1}^n{\cal A}_i$. Let us consider the function $\displaystyle u:\prod _{i=1}^n{\rm Con}({\cal A}_i)\rightarrow {\rm Con}({\cal A})$, defined by:$$u(\theta _1,\theta _2,\ldots ,\theta _n)=\{(x,y)\in A^2\ |\ \forall \, i\in \overline{1,n}\, (x_i,y_i)\in \theta _i\}=\theta _1\times \theta _2\ldots \times \theta _n$$for all $\theta _1\in {\rm Con}({\cal A}_1)$, $\theta _2\in {\rm Con}({\cal A}_2)$, $\ldots $, $\theta _n\in {\rm Con}({\cal A}_n)$, where, for each $a\in A$ and every $i\in \overline{1,n}$, we have denoted by $a_i\in A_i$ the $i$--th component of $a$. Also, let us denote, for every $i\in \overline{1,n}$, by $\pi _i:{\cal A}\rightarrow {\cal A}_i$ the canonical projection (for all $a\in A$, $\pi _i(a)=a_i$).

\begin{lemma}{\rm \cite{bj}}
$u$ is a lattice isomorphism.\label{ueizom}\end{lemma}

Now let us define $\displaystyle \lambda :{\cal A}/\theta \rightarrow \prod _{i=1}^n{\cal A}_i/\theta _i$, for all $a\in A$, $\lambda (a/\theta )=(a_1/\theta _1,\ldots ,a_n/\theta _n)$.

\begin{remark}
Clearly, $\lambda $ is an isomorphism of $\tau $--algebras.\label{lambdaeizom}\end{remark}

\begin{proposition}
\begin{enumerate}

\item\label{propproduse1} $\displaystyle {\cal A}(\varphi )=\prod _{i=1}^n{\cal A}_i(\varphi )$;
\item\label{propproduse2} ${\cal A}$ has $\varphi $--LP iff ${\cal A}_i$ has $\varphi $--LP, for every $i\in \overline{1,n}$.
\end{enumerate}\label{propproduse}\end{proposition}

\begin{proof} (\ref{propproduse1}) For simplicity, we shall assume that $\varphi (v)$ is of the form $\exists \, w\, \psi (v,w)$. For all $a\in A$, we have:

\begin{center}

\begin{tabular}{rll}
$a\in {\cal A}(\varphi )$ & iff & ${\cal A}\vDash \varphi (a)$\\ 
& iff & there exists $b\in A$ such that ${\cal A}\vDash \psi (a,b)$\\ 
& iff, & for all $i\in \overline{1,n}$, there exists $b_i\in A_i$ such that ${\cal A}_i\vDash \psi (a_i,b_i)$\\ 
& iff, & for all $i\in \overline{1,n}$, ${\cal A}_i\vDash \varphi (a_i)$\\ 
& iff, & for all $i\in \overline{1,n}$, $a_i\in {\cal A}_i(\varphi )$\\
& iff & $\displaystyle a\in \prod _{i=1}^n{\cal A}_i(\varphi )$
\end{tabular}
\end{center}

\noindent (\ref{propproduse2}) ``$\Rightarrow :$`` By Proposition \ref{propcaturi}.

\noindent ``$\Leftarrow :$`` Let $\theta \in {\rm Con}({\cal A})$ and $a\in A$ such that $a/\theta \in ({\cal A}/\theta)(\varphi )$. By Lemma \ref{ueizom}, $\theta =\theta _1\times \ldots \times \theta _n$, with $\theta _i\in {\rm Con}({\cal A}_i)$ for all $i\in \overline{1,n}$. Remark \ref{lambdaeizom} shows that $a_i\in ({\cal A}_i/\theta _i)(\varphi )$ for all $i\in \overline{1,n}$. Since each of the $\tau $--algebras ${\cal A}_1,\ldots ,{\cal A}_n$ has $\varphi $--LP, thus, for all $i\in \overline{1,n}$, there exists an $e_i\in {\cal A}_i(\varphi )$ such that $a_i/\theta _i=e_i/\theta _i$. Let $e=(e_1,\ldots ,e_n)$. Then, according to (\ref{propproduse1}), $e\in {\cal A}(\varphi )$. Notice that $\lambda (a/\theta )=\lambda (e/\theta )$, so that $a/\theta =e/\theta $ by Remark \ref{lambdaeizom}, hence $a/\theta \in {\cal A}(\varphi )/\theta $, thus $({\cal A}/\theta)(\varphi )={\cal A}(\varphi )/\theta $, which means that ${\cal A}$ has $\varphi $--LP.\end{proof}

\begin{remark}
Proposition \ref{propproduse}, applied to:

\begin{enumerate}
\item Example \ref{ringex} leads to the result in \cite{andcam} and \cite{hannic} which says that a finite direct product of commutative rings is clean iff each of those rings is clean.
\item Example \ref{resboolex} is exactly \cite[Proposition $5.4$]{ggcm}.
\item From Example \ref{blatex}, it is straightforward that, with $\varphi $ equalling the formula in that example, the following hold for any non--empty family $({\cal L}_i)_{i\in I}$ of bounded distributive lattices, with ${\cal L}$ denoting the direct product of this family: $\displaystyle {\cal L}=\prod _{i\in I}{\cal L}_i$:

\begin{itemize}
\item ${\cal L}$ has the $(\varphi ,{\rm Con}({\cal L}))$--LP iff, for all $i\in I$, ${\cal L}_i$ has the $(\varphi ,{\rm Con}({\cal L}_i))$--LP;
\item ${\cal L}$ has the $(\varphi ,{\rm Id}({\cal L}))$--LP iff, for all $i\in I$, ${\cal L}_i$ has the $(\varphi ,{\rm Id}({\cal L}_i))$--LP;
\item ${\cal L}$ has the $(\varphi ,{\rm Filt}({\cal L}))$--LP iff, for all $i\in I$, ${\cal L}_i$ has the $(\varphi ,{\rm Filt}({\cal L}_i))$--LP; that is: ${\cal L}$ has the BLP iff, for all $i\in I$, ${\cal L}_i$ has the BLP.\end{itemize}
\end{enumerate}\end{remark}

We recall (\cite[p. $52$]{bur}) that a congruence $\theta $ of a $\tau $--algebra ${\cal A}$ is called a {\em factor congruence} iff there exists a congruence $\theta ^*$ of ${\cal A}$ such that $\theta \cap \theta ^*=\Delta _A$ and $\theta \vee \theta ^*=A^2$. In this case, $(\theta ,\theta ^*)$ is called a {\em pair of factor congruences.}

\begin{corollary}
Let $\varphi $ be a formula of ${\bf L}_{\tau }$ with at most one free variable, ${\cal A}$ be a congruence--distributive $\tau $--algebra and $(\theta ,\theta ^*)$ be a pair of factor congruences. Then: ${\cal A}$ has $\varphi $--LP iff both ${\cal A}/{\theta }$ and ${\cal A}/{\theta ^*}$ have $\varphi $--LP.\end{corollary}

\begin{proof} By \cite[Theorem $7.5$, p. $52$]{bur}, ${\cal A}$ is isomorphic to ${\cal A}/{\theta }\times {\cal A}/{\theta ^*}$. Now apply Proposition \ref{propproduse}, (\ref{propproduse2}).\end{proof}

Throughout the rest of this section, let ${\cal A}$ be a residuated lattice (with underlying set $A$) and $\varphi (v)$ be an atomic formula of ${\bf L}_{\rm res}$, i. e. a formula of the form $t_1(v)\approx t_2(v)$, where $t_1(v)$ and $t_2(v)$ are terms with the variable $v$. Then:$${\cal A}(\varphi )=\{a\in A\ |\ t_1(a)=t_2(a)\}=\{a\in A\ |\ d(t_1(a),t_2(a))=1\}.$$

\begin{lemma}
For all $a\in A$, $a/[d(t_1^{\cal A}(a),t_2^{\cal A}(a)))\in {\cal A}/[d(t_1^{\cal A}(a),t_2^{\cal A}(a)))(\varphi )$.\label{l3.11}\end{lemma}

\begin{proposition}
The following are equivalent:

\begin{enumerate}
\item\label{formegalterm1} ${\cal A}$ has $\varphi $--LP;
\item\label{formegalterm2} for all $a\in A$, there exists $e\in {\cal A}(\varphi )$ such that $d(a,e)\in [d(t_1^{\cal A}(a),t_2^{\cal A}(a)))$.
\end{enumerate}\label{formegalterm}\end{proposition}

\begin{proof} (\ref{formegalterm1})$\Rightarrow $(\ref{formegalterm2}): Let $a\in A$. By Lemma \ref{l3.11}, it follows that there exists $e\in {\cal A}(\varphi )$ such that $a/[d(t_1^{\cal A}(a),t_2^{\cal A}(a)))=e/[d(t_1^{\cal A}(a),t_2^{\cal A}(a)))$, thus $d(a,e)/[d(t_1^{\cal A}(a),t_2^{\cal A}(a)))=d(a/[d(t_1^{\cal A}(a),t_2^{\cal A}(a))),e/[d(t_1^{\cal A}(a),t_2^{\cal A}(a))))=1$, so $d(a,e)\in [d(t_1^{\cal A}(a),t_2^{\cal A}(a)))$.

\noindent (\ref{formegalterm2})$\Rightarrow $(\ref{formegalterm1}): Let $F$ be an arbitrary filter of ${\cal A}$ and $a\in A$ such that $a/F\in ({\cal A}/F)(\varphi )$, thus $d(t_1^{\cal A}(a),t_2^{\cal A}(a))/F=d(t_1^{{\cal A}/F}(a/F),t_2^{{\cal A}/F}(a/F))=1$, which means that $d(t_1^{\cal A}(a),t_2^{\cal A}(a))\in F$, so $[d(t_1^{\cal A}(a),t_2^{\cal A}(a)))\subseteq F$. The hypothesis of this implication ensures us that there exists an $e\in {\cal A}(\varphi )$ such that $d(a,e)\in [d(t_1^{\cal A}(a),t_2^{\cal A}(a)))$. It follows that $d(a,e)\in F$, which means that $a/F=e/F\in {\cal A}(\varphi )/F$, therefore $({\cal A}/F)(\varphi )={\cal A}(\varphi )/F$, thus $F$ has $\varphi $--LP. So ${\cal A}$ has $\varphi $--LP.\end{proof}

\begin{corollary}{\rm \cite{ggcm}}
The following are equivalent:

\begin{enumerate}
\item\label{alnostru1} ${\cal A}$ has BLP;
\item\label{alnostru2} for all $a\in A$, there exists $e\in {\cal B}({\cal A})$ such that $d(a,e)\in [a\vee \neg \, a)$.
\end{enumerate}\label{alnostru}\end{corollary}

\begin{corollary}
The following are equivalent:

\begin{enumerate}
\item\label{corilp1} ${\cal A}$ has ILP;
\item\label{corilp2} for all $a\in A$, there exists an idempotent element $e\in A$ such that $d(a,e)\in [d(a,a^2))$.
\end{enumerate}\label{corilp}\end{corollary}

\begin{remark}
Similarly to the two previous corollaries, we can obtain that the following are equivalent:

\begin{enumerate}
\item\label{corrlp1} ${\cal A}$ has RLP;
\item\label{corrlp2} for all $a\in A$, there exists $e\in {\rm Reg}({\cal A})$ such that $d(a,e)\in [d(a,\neg \, \neg \, a))$.
\end{enumerate}

But here, for all $a\in A$, we may take $e=\neg \, \neg \, a\in {\rm Reg}({\cal A})$, and we have, trivially, $d(a,e)=d(a,\neg \, \neg \, a)\in [d(a,\neg \, \neg \, a))$. So this is just another way of verifying that any residuated lattice has RLP.\label{corrlp}\end{remark}

\section{Lifting Properties in Residuated Lattices}
\label{lpsreslat}

In this section we study from the algebraic point of view two lifting properties in residuated lattices, namely BLP and ILP. The main result of the present section is a characterization theorem for residuated lattices with BLP.

Throughout this section, $A$ will be an arbitrary residuated lattice, unless mentioned otherwise. In the rest of this paper, we let residuated lattices and all other algebraic structures be referred to by their underlying sets, in accordance to the convention made in Section \ref{preliminaries}.

In the following sections, we shall use the next remark without referencing it.

\begin{remark}
For every filter $F$ of a residuated lattice $A$, it is trivial that:

\begin{enumerate}
\item\label{cealincl1} ${\cal B}(A)/F\subseteq {\cal B}(A/F)$, thus $F$ has BLP iff ${\cal B}(A/F)\subseteq {\cal B}(A)/F$ iff ${\cal B}(p_F)$ is surjective, where ${\cal B}(p_F)$ is the image through the functor ${\cal B}$ of the canonical surjection $p_F:A\rightarrow A/F$ (see also \cite{ggcm});
\item\label{cealincl2} ${\cal I}(A)/F\subseteq {\cal I}(A/F)$, thus $F$ has ILP iff ${\cal I}(A/F)\subseteq {\cal I}(A)/F$.\end{enumerate}\label{cealincl}\end{remark}

\begin{corollary}
For every filter $F$ of a residuated lattice $A$:

\begin{enumerate}

\item\label{clar1} if ${\cal B}(A/F)=\{0/F,1/F\}$, then $F$ has BLP;
\item\label{clar2} if ${\cal I}(A/F)=\{0/F,1/F\}$, then $F$ has ILP.\end{enumerate}\label{clar}\end{corollary}

\begin{proof} By Remarks \ref{cealincl} and \ref{imme}, (\ref{imme1}).\end{proof}

\begin{proposition}

Let $A$ be a residuated lattice and $F$ be a filter of $A$. Then:

\begin{enumerate}
\item\label{propmeproud1} if ${\cal B}(A)/F=A/F$, then every filter of $A$ that includes $F$ has BLP and ILP;
\item\label{propmeproud2} if ${\cal I}(A)/F=A/F$, then every filter of $A$ that includes $F$ has ILP.\end{enumerate}\label{propmeproud}\end{proposition}

\begin{proof} (\ref{propmeproud1}): Let $G$ be a filter of $A$ such that $F\subseteq G$. Assume that ${\cal B}(A)/F=A/F$, which means that, for any $a\in A$, there exists an $e\in {\cal B}(A)$ such that $a/F=e/F$, that is $a\leftrightarrow e\in F$, hence $a\leftrightarrow e\in G$, so $a/G=e/G$, hence ${\cal B}(A)/G=A/G$, thus $A/G={\cal B}(A)/G\subseteq {\cal B}(A/G)\subseteq A/G$ by Remark \ref{cealincl}, (\ref{cealincl1}), hence ${\cal B}(A)/G={\cal B}(A/G)$, so $G$ has BLP in $A$. According to Lemma \ref{cbool}, (\ref{cbool3}) and (\ref{cbool0}), ${\cal B}(A)/F=A/F$ implies $A/F={\cal B}(A)/F\subseteq {\cal I}(A)/F\subseteq A/F$, hence ${\cal I}(A)/F=A/F$. The fact that $G$ also has ILP now follows from (\ref{propmeproud2}) below.

\noindent (\ref{propmeproud2}): Analogous to the proof of the statement on BLP from (\ref{propmeproud1}).\end{proof}

\begin{corollary}
If $A$ is a residuated lattice and $F$ is a filter of $A$ such that $A/F=\{0/F,1/F\}$, then every filter of $A$ that includes $F$ has BLP and ILP.\label{corutil}\end{corollary}

\begin{proof} By Proposition \ref{propmeproud}, (\ref{propmeproud1}), and Remark \ref{imme}, (\ref{imme1}), applied to the residuated lattice $A/F$.\end{proof}

\begin{corollary}
Let $A$ be a residuated lattice. Then:

\begin{enumerate}
\item\label{cormeproud0} the trivial filter and the improper filter have BLP and ILP in $A$;

\item\label{cormeproud1} if ${\cal B}(A)=A$, then $A$ has BLP and ILP, that is every Boolean algebra induces a residuated lattice with BLP and ILP;
\item\label{cormeproud2} if ${\cal I}(A)=A$, then $A$ has ILP, that is every $G\ddot{o}del$ algebra has ILP.\end{enumerate}\label{cormeproud}\end{corollary}

\begin{proof} We felt it useful to put together these results, although they are part of Examples \ref{resboolex} and \ref{residex} above; of course, we also need to use Remark \ref{imme}, (\ref{imme1}), to obtain the first statement. 

Here, (\ref{cormeproud0}) follows immediately from Remark \ref{cealincl}, while (\ref{cormeproud1}) and (\ref{cormeproud2}) are immediate consequences of Proposition \ref{propmeproud}.\end{proof}

In what follows, we shall use the fact that the filters $[1)=\{1\}$ and $[0)=A$ have BLP and ILP without referencing the previous two corollaries.

\begin{corollary}
If a residuated lattice $A$ has BLP, then, for all $e\in {\cal B}(A)$, the residuated lattice $[e)$ has BLP.\label{corggcm}\end{corollary}

\begin{proof} By \cite[Proposition $2.18$]{eu3} and \cite[Corollary $5.9$]{ggcm}.\end{proof}

\begin{proposition}
The following are equivalent for any residuated lattice $A$:

\begin{enumerate}
\item\label{propblp1} $A$ has BLP;
\item\label{propblp2} for all $x\in A$, there exists $e\in {\cal B}(A)$ such that $e\in [x)$ and $\neg \, e\in [\neg \, x)$;
\item\label{propblp3} for all $x,y\in A$ such that $x\odot y=0$, there exists $e\in {\cal B}(A)$ such that $e\in [x)$ and $\neg \, e\in [y)$;
\item\label{propblp4} given any natural $n\geq 2$, for all $x_1,\ldots ,x_n\in A$ such that $x_1\odot \ldots \odot x_n=0$, there exist $e_1,\ldots ,e_n\in {\cal B}(A)$ such that $\displaystyle \bigwedge _{i=1}^ne_i=0$, $e_i\vee e_j=1$ for all $i,j\in \overline{1,n}$ with $i\neq j$, and $e_i\in [x_i)$ for all $i\in \overline{1,n}$.\end{enumerate}\label{propblp}\end{proposition}

\begin{proof} The equivalence (\ref{propblp1})$\Leftrightarrow $(\ref{propblp2}) is \cite[Proposition $4.6$]{ggcm}, but also follows immediately from Corollary \ref{alnostru} above.

\noindent (\ref{propblp2})$\Rightarrow $(\ref{propblp3}): Let $x,y\in A$ such that $x\odot y=0$, that is $y\leq \neg \, x$ by Lemma \ref{latrez}, (\ref{latrez1}), hence $[\neg \, x)\subseteq [y)$ by the form of a principal filter. The hypothesis of this implication then shows that there exists $e\in {\cal B}(A)$ such that $e\in [x)$ and $\neg \, e\in [\neg \, x)\subseteq [y)$.

\noindent (\ref{propblp3})$\Rightarrow $(\ref{propblp2}): By Lemma \ref{latrez}, (\ref{latrez1}).

\noindent (\ref{propblp1}),(\ref{propblp3})$\Rightarrow $(\ref{propblp4}): We proceed by induction on $n$. For $n=2$, (\ref{propblp4}) is exactly (\ref{propblp3}). Now let us assume that the statement in (\ref{propblp4}) is true for an $n\in \N $, $n\geq 2$, $n$ arbitrary but fixed. Let $x_1,\ldots ,x_n,x_{n+1}\in A$ such that $x_1\odot \ldots \odot x_n\odot x_{n+1}=0$. The hypoythesis (\ref{propblp3}) shows that then there exists an $f\in {\cal B}(A)$ such that $f\in [x_1\odot \ldots \odot x_n)$ and $\neg \, f\in [x_{n+1})$. So there exists $k\in \N ^*$ such that $x_1^k\odot \ldots \odot x_n^k\leq f$ and $x_{n+1}^k\leq \neg \, f$ (see the form of a principal filter and Lemma \ref{latrez}, (\ref{latrez6})). Let us consider the elements $x_1^k\vee f,\ldots ,x_n^k\vee f\in [f)$. Lemma \ref{latrez}, (\ref{latrez9}), ensures us that $\displaystyle (x_1^k\vee f)\odot \ldots \odot (x_n^k\vee f)=\bigodot _{i=1}^n(x_i^k\vee f)=(\bigodot _{i=1}^nx_i^k)\vee f=f$. By the hypothesis (\ref{propblp1}) and Corollary \ref{corggcm}, the residuated lattice $[f)$ has BLP, so, by the induction hypothesis, it follows that there exist $e_1,\ldots ,e_n\in {\cal B}([f))$ such that $\displaystyle \bigwedge _{i=1}^ne_i=f$, $e_i\vee e_j=1$ for all $i,j\in \overline{1,n}$ with $i\neq j$, and $e_i\in [x_i^k\vee f)$ for all $i\in \overline{1,n}$. Let $e_{n+1}=\neg \, f\in {\cal B}(A)$ (see Lemma \ref{cbool}, (\ref{cbool0})). Then, according to Lemma \ref{latrez}, (\ref{latrez8}), $e_1\odot \ldots \odot e_n\odot e_{n+1}\leq (e_1\wedge \ldots \wedge e_n)\odot e_{n+1}=f\wedge \neg \, f=0$, thus $e_1\odot \ldots \odot e_n\odot e_{n+1}=0$. Also, for all $i\in \overline{1,n}$, $e_i\in [x_i^k\vee f)$, hence there exists an $m\in \N ^*$ such that, for all $i\in \overline{1,n}$, $(x_i^k\vee f)^m\leq e_i$ (see the form of a principal filter and Lemma \ref{latrez}, (\ref{latrez6})), so $x_i^{km}\leq e_i$, thus $e_i\in [x_i)$. Also, $e_{n+1}=\neg \, f\in [x_{n+1})$ and, for all $i\in \overline{1,n}$, $e_i\vee e_{n+1}\geq f\vee \neg \, f=1$ by Lemma \ref{cbool}, (\ref{cbool1}), so $e_i\vee e_{n+1}=1$.

\noindent (\ref{propblp4})$\Rightarrow $(\ref{propblp3}): Trivial.\end{proof}

\begin{remark}{\rm \cite{ggcm}}\begin{enumerate}
\item\label{celeloc1} Any local residuated lattice has BLP.
\item\label{celeloc2} Any linearly orderred residuated lattice is local.
\item\label{celeloc3} Consequently, any linearly orderred residuated lattice has BLP.
\item\label{celeloc4} Any prime filter of a residuated lattice has BLP.\end{enumerate}\label{celeloc}\end{remark}

\begin{remark}
(\ref{celeloc4}) from Remark \ref{celeloc} provides us with another way of proving that all residuated chains have BLP, by simply using the fact that all filters of a residuated chain are prime, a fact that follows immediately from the definition of a prime filter. And, just for the fun of the argument, there is an even simpler way of showing that all residuated chains have BLP: assume that $A$ is a chain and take $a\in A$. Then $a\odot \neg \, a=0$, according to Lemma \ref{latrez}, (\ref{latrez7}), and either $a\leq \neg \, a$ or $\neg \, a\leq a$, thus Lemma \ref{latrez}, (\ref{latrez6}), gives us $a^2=0$ or $(\neg \, a)^2=0$, hence $[a)=A$ or $[\neg \, a)=A$. By Lemma \ref{latrez}, (\ref{latrez5}), we have: if $[a)=A$, then $0\in [a)$ and, of course, $\neg \, 0=1\in [\neg \, a)$; if $[\neg \, a)=A$, then $\neg \, 1=0\in [\neg \, a)$ and, of course, $1\in [a)$. The fact that $0,1\in {\cal B}(A)$ and Proposition \ref{propblp} now show that the residuated chain $A$ has BLP.\end{remark}

\begin{proposition}
If $A$ is linearly orderred, then any principal filter generated by an idempotent element of $A$ has ILP.\label{pfiltidemp}\end{proposition}

\begin{proof} Assume that $A$ is a chain. Let $a\in {\cal I}(A)$ and $F=[a)$. Consider an $x\in A$ such that $x/F\in {\cal I}(A/F)$, which means that $x/F=(x/F)^2=x^2/F$, that is $x\leftrightarrow x^2\in F$, so that $x\rightarrow x^2\in F$ by Lemma \ref{latrez}, (\ref{latrez6}) and (\ref{latrez5}). Thus $x\rightarrow x^2\in [a)$, so $a\leq x\rightarrow x^2$, which is equivalent to $a\odot x\leq x^2$. Since $A$ is a chain, we have either $a\leq x$ or $x\leq a$. If $a\leq x$, then $x\in [a)=F$, so $x/F=1/F\in {\cal I}(A)/F$ since $1\in {\cal I}(A)$. If $x\leq a$, then $x^2\leq a\odot x$ by Lemma \ref{latrez}, (\ref{latrez6}). So, in this case, $x^2\leq a\odot x\leq x^2$, that is $x^2=a\odot x$. Then $x^4=(a\odot x)^2=a^2\odot x^2=a^2\odot a\odot x=a^3\odot x=a\odot x=x^2$, because $a\in {\cal I}(A)$. Thus $x^2=x^4=(x^2)^2$, which means that $x^2\in {\cal I}(A)$. Hence $x/F=x^2/F\in {\cal I}(A)/F$. Therefore ${\cal I}(A/F)\subseteq {\cal I}(A)/F$, thus $F$ has ILP according to Remark \ref{cealincl}, (\ref{cealincl2}).\end{proof}

\begin{remark}
If $A$ is finite, then ${\rm Filt}(A)={\rm PFilt}(A)=\{[a)\ |\ a\in {\cal I}(A)\}$. This is obvious, because all filters of a finite residuated lattice are finite and thus principal, while Lemma \ref{latrez}, (\ref{latrez6}), shows that, for every $a\in A$, the sequence $a,a^2,a^3,\ldots ,a^n,\ldots $ is decreasing, thus it is stationery in the case when $A$ is finite, so in this case there exists an $n\in \N ^*$ such that $a^n\in {\cal I}(A)$, which shows that the principal filter $[a)=[a^n)$ has an idempotent generator.\label{filtlrfin}\end{remark}

\begin{corollary}
Any finite residuated chain has ILP.\end{corollary}

\begin{proof} By Proposition \ref{pfiltidemp} and Remark \ref{filtlrfin}.\end{proof}

\begin{remark}
If $F$ is a filter of $A$ such that there exist $\min (F)\in A$, then $\min (F)\in {\cal I}(A)$ and $F=[\min (F))$. Indeed, if $a=\min (F)\in A$, then the fact that $F=[a)$ follows by double inclusion. Moreover, we get that $F=[a)=\{x\in A\ |\ a\leq x\}$, which clearly shows that $a$ is an idempotent element of $A$, because the fact that $a^2\in F$ implies $a\leq a^2$, which in turn implies $a=a^2$ by Lemma \ref{latrez}, (\ref{latrez6}).\label{filtcumin}\end{remark}

\begin{corollary}
Any well--orderred residuated lattice has ILP.\label{lrbineord}\end{corollary}

\begin{proof} Assume that $A$ is well orderred and let $F$ be an arbitrary filter of $A$. Then $F$ is non--empty, hence $F$ has a minimum in $A$, thus $F$ is a principal filter generated by an idempotent element of $A$, as Remark \ref{filtcumin} ensures us. It also follows that $A$ is a chain. Hence $F$ has ILP by Proposition \ref{pfiltidemp}, therefore $A$ has ILP.\end{proof}

\begin{example} The properties BLP and ILP do not coincide in every residuated lattice. Take, for instance, this example of residuated lattice from \cite{ior}: $A=\{0,a,b,c,1\}$, with the following Hasse diagram and table for the implication, and having $\odot =\wedge $:

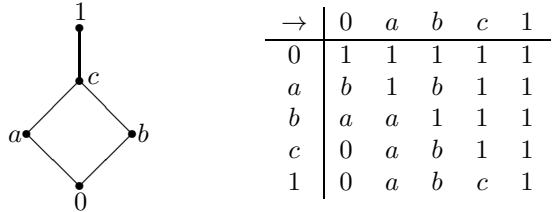
\begin{center}
\begin{tabular}{cc}
\begin{picture}(70,85)(0,0)
\put(30,10){\line(-1,1){20}}
\put(30,10){\line(1,1){20}}
\put(30,50){\line(-1,-1){20}}
\put(30,50){\line(1,-1){20}}
\put(30,50){\line(0,1){20}}
\put(30,10){\circle*{3}}
\put(10,30){\circle*{3}}
\put(50,30){\circle*{3}}
\put(30,50){\circle*{3}}
\put(30,70){\circle*{3}}
\put(28,1){$0$}
\put(3,27){$a$}
\put(52,27){$b$}
\put(33,49){$c$}
\put(28,73){$1$}
\end{picture}
&\hspace*{15pt}
\begin{picture}(70,85)(0,0)
\put(0,39){\begin{tabular}{c|ccccc}
$\rightarrow $ & $0$ & $a$ & $b$ & $c$ & $1$ \\ \hline
$0$ & $1$ & $1$ & $1$ & $1$ & $1$ \\
$a$ & $b$ & $1$ & $b$ & $1$ & $1$ \\
$b$ & $a$ & $a$ & $1$ & $1$ & $1$ \\
$c$ & $0$ & $a$ & $b$ & $1$ & $1$ \\
$1$ & $0$ & $a$ & $b$ & $c$ & $1$
\end{tabular}}
\end{picture}
\end{tabular}
\end{center}

As shown in \cite[Example $3.11$]{ggcm}, $A$ does not satisfy BLP, because ${\rm Rad}(A)=[c)=\{c,1\}$ does not have BLP. Since $\odot =\wedge $, it follows that ${\cal I}(A)=A$, hence $A$ has ILP, as shown by Example \ref{residex} or Corollary \ref{cormeproud}, (\ref{cormeproud2}).\label{exlpdif}\end{example}

\begin{example} Here is an example of a residuated lattice that is neither local, not even a chain, nor a Boolean algebra, in fact not even distributive, and it is neither a ${\rm G\ddot{o}del}$ algebra (nor an involutive residuated lattice, for that matter), but it has BLP and ILP (see Remark \ref{celeloc} and Corollary \ref{cormeproud} above).

This example of residuated lattice appears in \cite{gal}, \cite{ior}, \cite{kow}. It has $A=\{0,a,b,c,d,1\}$ as underlying set and the following Hasse diagram and operations $\odot $ and $\rightarrow $: 

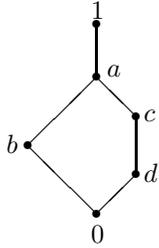
\begin{center}
\begin{tabular}{ccc}

\begin{picture}(100,70)(0,0)
\put(37,11){\circle*{3}}
\put(35,0){$0$}
\put(37,11){\line(1,1){15}}
\put(52,26){\circle*{3}}
\put(55,23){$d$}
\put(52,26){\line(0,1){22}}
\put(52,48){\circle*{3}}
\put(55,46){$c$}
\put(52,48){\line(-1,1){15}}
\put(37,63){\circle*{3}}
\put(41,63){$a$}
\put(37,11){\line(-1,1){26}}
\put(11,37){\circle*{3}}
\put(3,34){$b$}
\put(11,37){\line(1,1){26}}
\put(37,63){\line(0,1){20}}
\put(37,83){\circle*{3}}
\put(35,85){$1$}
\end{picture}
&
\begin{picture}(140,100)(0,0)
\put(0,45){
\begin{tabular}{c|cccccc}
$\odot $ & $0$ & $a$ & $b$ & $c$ & $d$ & $1$ \\ \hline
$0$ & $0$ & $0$ & $0$ & $0$ & $0$ & $0$ \\
$a$ & $0$ & $a$ & $b$ & $d$ & $d$ & $a$ \\
$b$ & $0$ & $b$ & $b$ & $0$ & $0$ & $b$ \\
$c$ & $0$ & $d$ & $0$ & $d$ & $d$ & $c$ \\
$d$ & $0$ & $d$ & $0$ & $d$ & $d$ & $d$ \\
$1$ & $0$ & $a$ & $b$ & $c$ & $d$ & $1$
\end{tabular}}
\end{picture}
&
\begin{picture}(120,100)(0,0)
\put(0,45){
\begin{tabular}{c|cccccc}

$\rightarrow $ & $0$ & $a$ & $b$ & $c$ & $d$ & $1$ \\ \hline
$0$ & $1$ & $1$ & $1$ & $1$ & $1$ & $1$ \\
$a$ & $0$ & $1$ & $b$ & $c$ & $c$ & $1$ \\
$b$ & $c$ & $1$ & $1$ & $c$ & $c$ & $1$ \\
$c$ & $b$ & $1$ & $b$ & $1$ & $a$ & $1$ \\
$d$ & $b$ & $1$ & $b$ & $1$ & $1$ & $1$ \\
$1$ & $0$ & $a$ & $b$ & $c$ & $d$ & $1$
\end{tabular}}
\end{picture}
\end{tabular}
\end{center}

$A/[a)=\{0/[a),c/[a)=d/[a),b/[a),a/[a)=1/[a)\}$ is the four--element Boolean algebra (the lozenge), hence ${\cal B}(A/[a))=A/[a)$, therefore every filter of $A$ which includes $[a)$ (that is every non--trivial filter of $A$) has BLP and ILP, as shown by Proposition \ref{propmeproud}, (\ref{propmeproud1}). $[1)$ has BLP and ILP. Therefore $A$ has BLP and ILP.

To show that $A$ has neither of the structures listed above, notice that: ${\rm Max}(A)=\{[b),[d)\}$, ${\cal B}(A)=\{0,1\}$, ${\cal I}(A)=A\setminus \{c\}$ and ${\rm Reg}(A)=A\setminus \{a\}$.\label{nice}\end{example}

\begin{example}
If $A$ is an MV--algebra, then $A$ can be organized as a residuated lattice having ${\cal B}(A)={\cal I}(A)$ (\cite{ior}). Evidently, this also holds for its quotient algebras. Thus:

\begin{itemize}
\item $A$ has BLP iff $A$ has ILP;
\item moreover, for any filter $F$ of $A$, we have: $F$ has BLP iff $F$ has ILP.\end{itemize}\label{mvalgle}\end{example}

\begin{example}
Here are three classical infinite examples: the residuated lattices induced by the three fundamental t--norms have BLP and ILP. Indeed, all three of them are chains, hence they all have BLP, according to Remark \ref{celeloc}, (\ref{celeloc3}).
 
The residuated lattice induced by the \L ukasiewicz t--norm is an MV--algebra, hence, since it has BLP, it also has ILP, as Example \ref{mvalgle} ensures us.

The residuated lattice induced by the ${\rm G\ddot{o}del}$ t--norm is a ${\rm G\ddot{o}del}$ algebra, hence it has ILP by Corollary \ref{cormeproud}, (\ref{cormeproud2}).

Now let $P$ be the residuated lattice induced by the product t--norm. It is clear that, for every $x\in P\setminus \{0,1\}$, $[x)=(0,1]=P\setminus \{0\}$ (the real interval $(0,1]$), and, moreover, for any $S\subseteq P\setminus \{0\}$ such that $S\neq \emptyset $ and $S\neq 1$, $[S)=(0,1]=P\setminus \{0\}$. Therefore, ${\rm Filt}(P)=\{P,P\setminus \{0\},[1)\}$. The filters $P$ and $[1)$ have ILP. As for the filter $P\setminus \{0\}$, clearly, $P/(P\setminus \{0\})=\{0/(P\setminus \{0\}),1/(P\setminus \{0\})\}$, hence $P\setminus \{0\}$ has ILP by Corollary \ref{corutil}. Thus $P$ has ILP.\end{example}

\section{Spectra and the Reticulation}
\label{specretic}

In this section we present notions and properties related to the prime and maximal spectrum of a residuated lattice. By using the reticulation functor from \cite{eu3}, \cite{eu1}, \cite{eu}, \cite{eu2}, \cite{eu4}, \cite{eu5}, \cite{eu7}, we establish the link between BLP in residuated lattices and BLP in their reticulations.

The first results in this section, which we present without their proofs, are known and straightforward. For a further study of these results, we refer the reader to \cite{bur}, \cite{lciu}, \cite{gal}, \cite{eu3}, \cite{haj}, \cite{kow}, \cite{eu1}, \cite{eu}, \cite{eu2}, \cite{eu4}, \cite{eu5}, \cite{eu7}, \cite{pic}. For the variants of these results in the case of BL--algebras, \cite{leo} can be consulted, while, for their versions in the case of MV--algebras, \cite{cdom} and \cite{ile} should be studied.

We shall be keeping the following notations, for any residuated lattice $A$, every $a\in A$ and each $F\in {\rm Filt}(A)$:

\begin{center}
\begin{tabular}{ll}
$V(F)=\{P\in {\rm Spec}(A)\ |\ F\subseteq P\}$ & $D(F)=\{P\in {\rm Spec}(A)\ |\ F\nsubseteq P\}$\\ 
$V(a)=\{P\in {\rm Spec}(A)\ |\ a\in P\}$ & $D(a)=\{P\in {\rm Spec}(A)\ |\ a\notin P\}$\\ 
$v(F)=V(F)\cap {\rm Max}(A)=\{P\in {\rm Max}(A)\ |\ F\subseteq P\}$ & $d(F)=D(F)\cap {\rm Max}(A)=\{P\in {\rm Max}(A)\ |\ F\nsubseteq P\}$\\ 
$v(a)=V(a)\cap {\rm Max}(A)=\{P\in {\rm Max}(A)\ |\ a\in P\}$ & $d(a)=D(a)\cap {\rm Max}(A)=\{P\in {\rm Max}(A)\ |\ a\notin P\}$
\end{tabular}\end{center}

Everywhere in the rest of this section, $A$ will be an arbitrary residuated lattice, unless mentioned otherwise.

For all $F,G\in {\rm Filt}(A)$ and all $a,b\in A$:

\begin{itemize}
\item $D(F)={\rm Spec}(A)\setminus V(F)$, $d(F)={\rm Max}(A)\setminus v(F)$, $D(a)={\rm Spec}(A)\setminus V(a)$, $d(a)={\rm Max}(A)\setminus v(a)$;
\item $V([a))=V(a)$, $v([a))=v(a)$, $D([a))=D(a)$, $d([a))=d(a)$;
\item $V(F)={\rm Spec}(A)$ iff $F=\{1\}$; $V(F)=\emptyset $ iff $F=A$;
\item $V(a)={\rm Spec}(A)$ iff $a=1$; $V(a)=\emptyset $ iff $a$ is nilpotent;
\item $D(F)={\rm Spec}(A)$ iff $F=A$; $D(F)=\emptyset $ iff $F=\{1\}$;
\item $D(a)={\rm Spec}(A)$ iff $a$ is nilpotent; $D(a)=\emptyset $ iff $a=1$;
\item $v(F)={\rm Max}(A)$ iff $F\subseteq {\rm Rad}(A)$; $v(F)=\emptyset $ iff $F=A$;
\item $v(a)={\rm Max}(A)$ iff $a\in {\rm Rad}(A)$; $v(a)=\emptyset $ iff $a$ is nilpotent;
\item $d(F)={\rm Max}(A)$ iff $F=A$; $d(F)=\emptyset $ iff $F\subseteq {\rm Rad}(A)$;
\item $d(a)={\rm Max}(A)$ iff $a$ is nilpotent; $d(a)=\emptyset $ iff $a\in {\rm Rad}(A)$;
\item if $a\in F$, then: $V(F)\subseteq V(a)$, $v(F)\subseteq v(a)$, $D(a)\subseteq D(F)$, $d(a)\subseteq d(F)$;
\item if $F\subseteq G$, then: $V(G)\subseteq V(F)$, $v(G)\subseteq v(F)$, $D(F)\subseteq D(G)$, $d(F)\subseteq d(G)$;
\item if $a\leq b$, then: $V(a)\subseteq V(b)$, $v(a)\subseteq v(b)$, $D(b)\subseteq D(a)$, $d(b)\subseteq d(a)$.\end{itemize}

\begin{proposition}
Let $F$ and $G$ be filters of $A$ and $(F_i)_{i\in I}$ be a non--empty family of filters of $A$. Then:

\begin{enumerate}
\item\label{vsid2} $V(F\cap G)=V(F)\cup V(G)$, $D(F\cap G)=D(F)\cap D(G)$, $v(F\cap G)=v(F)\cup v(G)$, $d(F\cap G)=d(F)\cap d(G)$;
\item\label{vsid1} $\displaystyle V(\bigvee _{i\in I}F_i)=\bigcap _{i\in I}V(F_i)$, $\displaystyle D(\bigvee _{i\in I}F_i)=\bigcup _{i\in I}D(F_i)$, $\displaystyle v(\bigvee _{i\in I}F_i)=\bigcap _{i\in I}v(F_i)$, $\displaystyle d(\bigvee _{i\in I}F_i)=\bigcup _{i\in I}d(F_i)$.\end{enumerate}\label{vsid}\end{proposition}

\begin{proposition}
Let $a,b\in A$. Then:

\begin{enumerate}
\item\label{vsidelem2} $V(a\vee b)=V(a)\cup V(b)$, $D(a\vee b)=D(a)\cap D(b)$, $v(a\vee b)=v(a)\cup v(b)$, $d(a\vee b)=d(a)\cap d(b)$;
\item\label{vsidelem1} $V(a\odot b)=V(a\wedge b)=V(a)\cap V(b)$, $D(a\odot b)=D(a\wedge b)=D(a)\cup D(b)$, $v(a\odot b)=v(a\wedge b)=v(a)\cap v(b)$, $d(a\odot b)=d(a\wedge b)=d(a)\cup d(b)$.\end{enumerate}\label{vsidelem}\end{proposition}

\begin{proposition}
For any filter $F$ of $A$, the following equalities hold: $\displaystyle V(F)=\bigcap _{a\in F}V(a)$, $\displaystyle v(F)=\bigcap _{a\in F}v(a)$, $\displaystyle D(F)=\bigcup _{a\in F}D(a)$, $\displaystyle d(F)=\bigcup _{a\in F}d(a)$.\label{guguloi}\end{proposition}

Since each filter of $A$ is the intersection of the prime filters that include it, we have:

\begin{lemma}
For any $F,G\in {\rm Filt}(A)$ and any $a,b\in A$:

\begin{itemize}
\item $D(F)=D(G)$ iff $F=G$; $D(F)\subseteq D(G)$ iff $F\subseteq G$; $D(a)=D(b)$ iff $[a)=[b)$; $D(a)\subseteq D(b)$ iff $[a)\subseteq [b)$ iff $a\in [b)$;
\item $V(F)=V(G)$ iff $F=G$; $V(F)\subseteq V(G)$ iff $G\subseteq F$; $V(a)=V(b)$ iff $[a)=[b)$; $V(a)\subseteq V(b)$ iff $[b)\subseteq [a)$ iff $b\in [a)$.\end{itemize}\label{incld}\end{lemma}

$\{D(F)\ |\ F\in {\rm Filt}(A)\}$ is a topology on ${\rm Spec}(A)$, called {\em the Stone topology (of ${\rm Spec}(A)$).} The Stone topology induces on ${\rm Max}(A)$ the topology $\{d(F)\ |\ F\in {\rm Filt}(A)\}$, called {\em the Stone topology of ${\rm Max}(A)$.} Throughout the rest of this paper, ${\rm Spec}(A)$ and ${\rm Max}(A)$ shall be considerred as topological spaces with the Stone topology. Trivially, the closed sets of the Stone topology of ${\rm Spec}(A)$ are $\{V(F)\ |\ F\in {\rm Filt}(A)\}$ and the closed sets of the Stone topology of ${\rm Max}(A)$ are $\{v(F)\ |\ F\in {\rm Filt}(A)\}$. Also, as shown by Proposition \ref{guguloi}, $\{D(a)\ |\ a\in A\}$ is a basis (of open sets) for the Stone topology of ${\rm Spec}(A)$ and $\{d(a)\ |\ a\in A\}$ is a basis (of open sets) for the Stone topology of ${\rm Max}(A)$.
 
\begin{lemma}
For every $e\in {\cal B}(A)$: $D(e)=V(\neg \, e)$, $d(e)=v(\neg \, e)$, $V(e)=D(\neg \, e)$, $v(e)=d(\neg \, e)$.\label{vsidbool}\end{lemma}

\begin{proof} By the definition of a prime filter, along with Lemma \ref{cbool}, (\ref{cbool1}) and (\ref{cbool0}).\end{proof}

\begin{corollary}
\begin{itemize}
\item $\{V(e)\ |\ e\in {\cal B}(A)\}=\{D(e)\ |\ e\in {\cal B}(A)\}$ is a family of clopen sets in ${\rm Spec}(A)$;
\item $\{v(e)\ |\ e\in {\cal B}(A)\}=\{d(e)\ |\ e\in {\cal B}(A)\}$ is a family of clopen sets in ${\rm Max}(A)$.\end{itemize}\label{corvdb}\end{corollary}

\begin{lemma}\begin{itemize}
\item The topological closure in ${\rm Spec}(A)$ of any subset $S$ of ${\rm Spec}(A)$ is $\displaystyle V(\bigcap _{P\in S}P)$.
\item Consequently, the topological closure of ${\rm Max}(A)$ in ${\rm Spec}(A)$ is $V({\rm Rad}(A))$.
\item Furthermore, if $A$ is semisimple, then the topological closure of ${\rm Max}(A)$ in ${\rm Spec}(A)$ is ${\rm Spec}(A)$, that is ${\rm Max}(A)$ is dense in ${\rm Spec}(A)$.\end{itemize}\label{lema1}\end{lemma}

\begin{lemma}
Let $a\in A$ and $e\in {\cal B}(A)$. Then:

\begin{enumerate}
\item\label{ltextnou1} $v(a)\subseteq v(e)$ iff $a\odot \neg \, e$ is nilpotent;
\item\label{ltextnou2} $v(a)\subseteq d(e)$ iff $a\odot e$ is nilpotent;
\item\label{ltextnou3} $d(a)\subseteq v(e)$ iff $a\vee e\in {\rm Rad}(A)$.\end{enumerate}\label{ltextnou}\end{lemma}

\begin{proof} (\ref{ltextnou1}) By Lemma \ref{vsidbool} and Proposition \ref{vsidelem}, $v(a)\subseteq v(e)$ iff $v(a)\subseteq {\rm Max}(A)\setminus v(\neg \, e)$ iff $v(a)\cap v(\neg \, e)=\emptyset $ iff $v(a\odot \neg \, e)=\emptyset $ iff $a\odot \neg \, e$ is nilpotent.

\noindent (\ref{ltextnou2}) By (\ref{ltextnou1}) and Lemma \ref{cbool}, (\ref{cbool0}).

\noindent (\ref{ltextnou3}) Again by Proposition \ref{vsidelem}, $d(a)\subseteq v(e)$ iff $d(a)\subseteq {\rm Max}(A)\setminus d(e)$ iff $d(a)\cap d(e)=\emptyset $ iff $d(a\vee e)=\emptyset $ iff $a\vee e\in {\rm Rad}(A)$.\end{proof}

\begin{lemma}
For any $a\in A$, $\displaystyle d(a)=\bigcup _{n=1}^{\infty }v(\neg \, a^n)$.\label{l4textnou}\end{lemma}

\begin{proof} Let $a\in A$ and $M\in {\rm Max}(A)$. Then, by Lemma \ref{caractfiltmax}: $M\in d(a)$ iff $a\notin M$ iff $\neg \, a^n\in M$ for some $n\in \N ^*$ iff $M\in v(\neg \, a^n)$ for some $n\in \N ^*$ iff $\displaystyle M\in \bigcup _{n=1}^{\infty }v(\neg \, a^n)$.\end{proof}

\begin{lemma} The topological spaces ${\rm Max}(A)$ and ${\rm Max}(A/{\rm Rad}(A))$ are homeomorphic.\label{lema2}\end{lemma}

The Stone topology of the prime (and that of the maximal) spectrum of a bounded distributive lattice is defined just as described above for a residuated lattice, and it satisfies all the properties stated above for residuated lattices, with one small modification: $\odot $ must be replaced by $\wedge $ in each of these properties.

\begin{definition}{\rm \cite{eu1}} A {\em reticulation of $A$} is a pair $(L(A),\lambda )$, where $L(A)$ is a bounded distributive lattice and $\lambda :A\rightarrow L(A)$ is a function that satisfies conditions (1)--(5) below:

\begin{flushleft}
\begin{tabular}{cl}
$(1)$ & for all $a,b\in A$, $\lambda (a\odot b)=\lambda (a)\wedge \lambda (b)$;\\ 
$(2)$ & for all $a,b\in A$, $\lambda (a\vee b)=\lambda (a)\vee \lambda (b)$;\\ 
$(3)$ & $\lambda (0)=0$; $\lambda (1)=1$;\\ 
$(4)$ & $\lambda $ is surjective;\\ 
$(5)$ & for all $a,b\in A$, $\lambda (a)\leq \lambda (b)$ iff there exists an $n\in \N ^{*}$ such that $a^{n}\leq b$.
\end{tabular}
\end{flushleft}
\label{reticulatia}
\end{definition}

\begin{proposition}{\rm \cite{eu1}} The reticulation of $A$ exists and it is unique up to a bounded lattice isomorphism. More precisely:

\begin{itemize}
\item there exists a reticulation of $A$;
\item if $(L_{1},\lambda _{1})$, $(L_{2},\lambda _{2})$ are two reticulations of $A$, then there exists an isomorphism of bounded lattices $f:L_{1}\rightarrow L_{2}$ such that $f\circ \lambda _{1}=\lambda _{2}$.\end{itemize}\label{exsiunic}\end{proposition}

We shall keep the notations from Definition \ref{reticulatia} in what follows. So, henceforth, $(L(A),\lambda )$ will be the reticulation of $A$. We shall also denote by $\lambda ^*$ the inverse image of the function $\lambda $: for all $S\subseteq L(A)$, $\lambda ^*(S)=\lambda ^{-1}(S)$.

\begin{proposition}{\rm \cite{eu1}, \cite{eu2}}

\begin{enumerate}
\item\label{retic0} For all $a,b\in A$, $\lambda (a\wedge b)=\lambda (a)\wedge \lambda (b)$, hence $\lambda $ is a bounded lattice morphism between the underlying bounded lattice of $A$ and $L(A)$; consequently, $\lambda $ is order--preserving (which also follows from the converse implication in condition $(5)$ for $n=1$);
\item\label{retic0,5} for all $a,b\in A$, $\lambda (a)=\lambda (b)$ iff $[a)=[b)$;
\item\label{retic0,7} for all $a\in A$ and all $n\in \N ^*$, $\lambda (a^n)=\lambda (a)$;
\item\label{retic1} $\lambda ^*\mid _{\textstyle {\rm Filt}(L(A))}:{\rm Filt}(L(A))\rightarrow {\rm Filt}(A)$ is a bounded lattice isomorphism (thus an order isomorphism), whose inverse takes every $F\in {\rm Filt}(A)$ to $\lambda (F)\in {\rm Filt}(L(A))$;
\item\label{retic2} $\lambda ^*\mid _{\textstyle {\rm Spec}(L(A))}:{\rm Spec}(L(A))\rightarrow {\rm Spec}(A)$ is a homeomorphism;
\item\label{retic3} $\lambda ^*\mid _{\textstyle {\rm Max}(L(A))}:{\rm Max}(L(A))\rightarrow {\rm Max}(A)$ is a homeomorphism;
\item\label{retic4} $\lambda \mid _{\textstyle {\cal B}(A)}:{\cal B}(A)\rightarrow {\cal B}(L(A))$ is a Boolean isomorphism;
\item\label{retic5} for every filter $F$ of $A$, if $(L(A/F),\lambda _F)$ is the reticulation of $A/F$, then the mapping that takes, for every $a\in A$, $\lambda _F(a/F)$ to $\lambda (a)/\lambda (F)$, is a well--defined bounded lattice isomorphism between $L(A/F)$ and $L(A)/\lambda (F)$.
\end{enumerate}\label{retic}\end{proposition}

\begin{remark}
The definition of $\lambda ^*$ (as the inverse image of another function) shows that $\lambda ^*$ preserves arbitrary intersections.\label{intersarb}\end{remark}

\begin{corollary}
\begin{enumerate}
\item\label{transferad1} $A$ is local iff $L(A)$ is local;
\item\label{transferad2} $A$ is semilocal iff $L(A)$ is semilocal;
\item\label{transferad3} $\lambda ^*({\rm Rad}(L(A)))={\rm Rad}(A)$, hence $\lambda ({\rm Rad}(A))={\rm Rad}(L(A))$.\label{transferad}\end{enumerate}\end{corollary}

\begin{proof} (\ref{transferad1}) and (\ref{transferad2}) follow from Proposition \ref{retic}, (\ref{retic3}).

\noindent (\ref{transferad3}) Proposition \ref{retic}, (\ref{retic3}), and Remark \ref{intersarb} give us the first equality, from which the second follows by Proposition \ref{retic}, (\ref{retic1}).\end{proof}

For any residuated lattice $B$, whose reticulation is $(L(B),\mu )$, let us denote ${\cal L}(B)=L(B)$. If $B$ and $C$ are arbitrary residuated lattices, whose reticulations are $(L(B),\mu )$ and $(L(C),\nu )$, respectively, and $f:B\rightarrow C$ is a morphism of residuated lattices, then let us denote by ${\cal L}(f):{\cal L}(B)\rightarrow {\cal L}(C)$ the mapping defined this way: for all $b\in B$, ${\cal L}(f)(\mu (b))=\nu (f(b))$. The surjectivity of $\mu $ (see condition $(4)$ above) ensures us that ${\cal L}(f):{\cal L}(B)\rightarrow {\cal L}(C)$ is completely defined by the equality above.

\begin{proposition}{\rm \cite{eu1}, \cite{eu2}}
\begin{enumerate}
\item\label{fctor1} ${\cal L}(f)$ above is well defined and it is a bounded lattice morphism;
\item\label{fctor2} with the definition above, ${\cal L}$ becomes a covariant functor from the category of residuated lattices to the category of bounded distributive lattices.
\end{enumerate}\label{fctor}\end{proposition}

${\cal L}$ is called {\em the reticulation functor} (\cite{eu1}).

\begin{remark}
If we denote by $\lambda :A\rightarrow {\rm PFilt}(A)$ the function defined by: $\lambda (a)=[a)$ for all $a\in A$, then $(({\rm PFilt}(A),\cap ,\vee ,A,[1)),\lambda )$ (the dual of the bounded distributive lattice ${\rm PFilt}(A)$, with the canonical surjection from $A$ to ${\rm PFilt}(A)$) is a reticulation of $A$.

For this ane one more construction of the reticulation of a residuated lattice, see \cite{eu1}.\label{constrretic}\end{remark}

\begin{proposition}
Let $F$ be a filter of $A$. Then the following are equivalent:

\begin{enumerate}
\item\label{blpretic1} the filter $F$ has BLP in $A$;
\item\label{blpretic2} the filter $\lambda (F)$ has BLP in ${\cal L}(A)$.
\end{enumerate}\label{blpretic}\end{proposition}

\begin{proof} Let $({\cal L}(A),\lambda )$ be the reticulation of $A$, $({\cal L}(A/F),\lambda _F)$ be the reticulation of $A/F$, and let $p_F:A\rightarrow A/F$ and $p_{\lambda (F)}:{\cal L}(A)\rightarrow {\cal L}(A)/\lambda (F)$ be the canonical surjections. Let $\varphi :{\cal L}(A/F)\rightarrow {\cal L}(A)/\lambda (F)$, defined by: for every $a\in A$, $\varphi (\lambda _F(a/F))=\lambda (a)/\lambda (F)$. According to Proposition \ref{retic}, (\ref{retic5}), $\varphi $ is well--defined and it is a bounded lattice isomorphism, hence its image through the functor ${\cal B}$ is a Boolean isomorphism. The fact that ${\cal L}$ is a covariant functor ensures us that the following diagrams are commutative:

\begin{center}
\begin{picture}(180,55)(0,0)
\put(20,35){$A$}
\put(80,35){${\cal L}(A)$}
\put(28,39){\vector(1,0){49}}
\put(49,41){$\lambda $}
\put(34,9){\vector(1,0){37}}
\put(90,32){\vector(0,-1){18}}
\put(103,36){\vector(3,-1){60}}
\put(49,12){$\lambda _F$}
\put(23,33){\vector(0,-1){20}}
\put(10,23){$p_F$}
\put(134,30){$p_{\lambda (F)}$}
\put(63,23){${\cal L}(p_F)$}
\put(14,5){$A/F$}
\put(74,5){${\cal L}(A/F)$}
\put(140,5){${\cal L}(A)/\lambda (F)$}
\put(110,9){\vector(1,0){28}}
\put(121,10){$\sim $}
\put(122,17){$\varphi $}
\end{picture}
\end{center}

If we apply the two functors ${\cal B}$ (the one from the category of residuated lattices to the category of Boolean algebras to the left of the diagram above in the left side, and the one from the category of bounded distributive lattices to the category of Boolean algebras to all of the diagram above in the right side, we get the following commutative diagrams in the category of Boolean algebras, in which $\lambda \mid _{\textstyle {\cal B}(A)}$ and $\lambda _F\! \mid _{\textstyle {\cal B}(A/F)}$ are Boolean isomorphisms, according to Proposition \ref{retic}, (\ref{retic4}): 

\begin{center}
\begin{picture}(250,55)(0,0)
\put(24,35){${\cal B}(A)$}
\put(120,35){${\cal B}({\cal L}(A))$}
\put(47,39){\vector(1,0){70}}
\put(59,45){$\lambda \mid _{\textstyle {\cal B}(A)}$}
\put(54,9){\vector(1,0){54}}
\put(135,32){\vector(0,-1){18}}
\put(157,34){\vector(3,-1){60}}

\put(49,16){$\lambda _F\! \mid _{\textstyle {\cal B}(A/F)}$}
\put(30,33){\vector(0,-1){20}}
\put(0,23){${\cal B}(p_F)$}
\put(177,30){${\cal B}(p_{\lambda (F)})$}
\put(93,23){${\cal B}({\cal L}(p_F))$}
\put(18,5){${\cal B}(A/F)$}
\put(110,5){${\cal B}({\cal L}(A/F))$}
\put(180,5){${\cal B}({\cal L}(A)/\lambda (F))$}
\put(160,9){\vector(1,0){19}}
\put(166,10){$\sim $}
\put(162,17){${\cal B}(\varphi )$}
\end{picture}

\end{center}

The commutativity of these diagrams shows that ${\cal B}(\varphi )\circ \lambda _F\! \mid _{\textstyle {\cal B}(A/F)}\circ \, \, {\cal B}(p_F)={\cal B}(p_{\lambda (F)})\circ \lambda \mid _{\textstyle {\cal B}(A)}$. This equality and the fact that ${\cal B}(\varphi )$, $\lambda _F\! \mid _{\textstyle {\cal B}(A/F)}$ and $\lambda \mid _{\textstyle {\cal B}(A)}$ are bijective show that ${\cal B}(p_F)$ is surjective iff ${\cal B}(p_{\lambda (F)})$ is surjective, hence $F$ has BLP in $A$ iff $\lambda (F)$ has BLP in ${\cal L}(A)$, according to Remark \ref{cealincl}, (\ref{cealincl1}).\end{proof}

\begin{proposition}
The following are equivalent:

\begin{enumerate}
\item\label{blptotretic1} $A$ has BLP;

\item\label{blptotretic2} ${\cal L}(A)$ has BLP.\end{enumerate}\label{blptotretic}\end{proposition}

\begin{proof} By Proposition \ref{retic}, (\ref{retic1}), and Proposition \ref{blpretic}.\end{proof}

\begin{remark}
Proposition \ref{blptotretic} shows that the reticulation functor ${\cal L}$ both preserves and reflects BLP.\end{remark}

The following two results are \cite[Lemma $5.1.4$ and Proposition $5.1.5$]{eu}. 

\begin{lemma}
Let $({\cal L}(A),\lambda )$ be the reticulation of $A$ and $a\in A$. Then: $a$ is archimedean iff $\lambda (a)\in {\cal B}({\cal L}(A))$.\label{elemarhim}\end{lemma}

\begin{proof}
``$\Rightarrow $:`` If $a$ is archimedean, then there exists an $n\in \N ^*$ such that $a^n\in {\cal B}(A)$, hence, by Proposition \ref{retic}, (\ref{retic0,7}), $\lambda (a)=\lambda (a^n)\in {\cal B}({\cal L}(A))$.

\noindent ``$\Leftarrow $:`` Conversely, let us assume that $\lambda (a)\in {\cal B}({\cal L}(A))$. Proposition \ref{retic}, (\ref{retic4}), ensures us that there exists $b\in {\cal B}(A)$ such that $\lambda (a)=\lambda (b)$. By condition $(5)$, $\lambda (a)=\lambda (b)$ iff $\lambda (a)\leq \lambda (b)$ and $\lambda (b)\leq \lambda (a)$ iff $a^n\leq b$ and $b^k\leq a$ for some $n,k\in \N ^*$. By Lemma \ref{cbool}, (\ref{cbool3}), and Lemma \ref{latrez}, (\ref{latrez6}), $b^k=b$, hence $b\leq a$, so $b=b^n\leq a^n$, therefore $a^n\leq b$ and $b\leq a^n$, so $a^n=b\in {\cal B}(A)$, thus $a$ is archimedean.\end{proof}

\begin{proposition}
$A$ is hyperarchimedean iff ${\cal L}(A)$ is a Boolean algebra.\label{hiperarhbool}\end{proposition}

\begin{proof}
By Lemma \ref{elemarhim} and the surjectivity of $\lambda $ in the reticulation $({\cal L}(A),\lambda )$ of $A$.\end{proof}

The next corollary is part of \cite[Corollary $4.18$]{ggcm}, but here we shall obtain it by means of the reticulation.

\begin{corollary}
Any hyperarchimedean residuated lattice has BLP.\end{corollary}

\begin{proof} By Proposition \ref{hiperarhbool}, $A$ is hyperarchimedean iff ${\cal L}(A)$ is a Boolean algebra and hence a bounded distributive lattice with BLP according to Example \ref{blatex}. The result now follows from Proposition \ref{blptotretic}.\end{proof}

\section{Topological Characterization for the Boolean Lifting Property}
\label{topchar}

In this section we study topological properties related to BLP in residuated lattices. The main results we obtain here are two characterization theorems for residuated lattices with BLP: Theorems \ref{p5.2} and \ref{p5.3}. Theorem \ref{p5.2} below corresponds to \cite[Theorem $3.10$]{sepe}, while Theorem \ref{p5.3} has been inspired by \cite[Theorem $1.7$]{mcgov} and the idea for Theorem \ref{p5.3} has also originated in \cite{mcgov}, in a discussion that starts on page $247$. It turns out that the study of residuated lattices with BLP is closely related to a larger class of residuated lattices, namely the ones with Gelfand property: any prime filter is included in a unique maximal filter. 

For any topological space $X$, we shall denote by ${\rm Clp}(X)$ the family of all clopen sets of $X$.

Throughout the rest of this section, $A$ will be an arbitrary residuated latice and $X$ will be an arbitrary topological space, unless mentioned otherwise.

The previous results hold if we replace $A$ by a bounded distributive lattice and $\odot $ by $\wedge $.

We recall that $X$ is said to be:

\begin{itemize}
\item {\em $T_0$} iff it satisfies the {\em $T_0$--separation axiom:} for any $x,y\in X$ with $x\neq y$, there exists an open set $D$ of $X$ such that either $x\in D$ and $y\notin D$, or $y\in D$ and $x\notin D$;
\item {\em $T_1$} iff, for every $x\in X$, the set $\{x\}$ is closed in $X$;
\item {\em zero--dimensional} iff $X$ has a basis of clopen sets; 
\item {\em strongly zero--dimensional} iff, given any closed set $T$ and any open set $V$ of $X$ with $T\subseteq V$, there exists a clopen set $U$ of $X$ such that $T\subseteq U\subseteq V$ (according to \cite{sepe});
\item {\em normal} iff any two disjoint closed sets of $X$ can be separated through open sets of $X$ (that is, for any two closed sets $C$ and $D$ of $X$ such that $C\cap D=\emptyset $, there exist two open sets $U$ and $V$ of $X$ such that $C\subseteq U$, $D\subseteq V$ and $U\cap V=\emptyset $);
\item {\em Boolean} iff $X$ is compact, Hausdorff and zero--dimensional.\end{itemize}

Clearly, any $T_1$ compact normal space is Hausdorff, thus any $T_1$ compact normal zero--dimensional space is Boolean.

${\rm Spec}(A)$ is a $T_0$ compact topological space, while ${\rm Max}(A)$ is a $T_1$ compact topological space, with the Stone topology. In the particular case when $A$ is a BL--algebra, ${\rm Max}(A)$ is a compact Hausdorff space.

\begin{proposition}{\rm \cite[Lemma $2.5$]{luyu}, \cite[Lemma $3.8$]{sepe}} The following are equivalent:

\begin{enumerate}
\item\label{multcit1} $X$ is strongly zero--dimensional;
\item\label{multcit2} any two disjoint closed sets of $X$ can be separated through clopen sets of $X$;
\item\label{multcit3} if there exist two open sets $U$ and $V$ of $X$ such that $X=U\cup V$, then there exist two clopen sets $C$ and $D$ of $X$ such that $C\subseteq U$, $D\subseteq V$, $C\cap D=\emptyset $ and $C\cup D=X$.
\end{enumerate}\label{multcit}\end{proposition}

An obvious consequence of the previous proposition is the fact that any strongly zero--dimensional topological space is normal. According to \cite{sepe}, a $T_1$ compact space is strongly zero--dimensional iff it is zero--dimensional, hence, furthermore: a $T_1$ compact space is zero--dimensional iff it is strongly zero--dimensional iff it is normal iff it is Boolean. Therefore:

\begin{corollary}
The topological space ${\rm Max}(A)$ is zero--dimensional iff it is strongly zero--dimensional iff it is normal iff it is Boolean.\label{coolmix}\end{corollary}

Let $L$ be a bounded distributive lattice. Then $L$ is said to be:

\begin{itemize}
\item {\em normal} iff, for all $x,y\in L$, if $x\vee y=1$, then there exist $u,v\in L$ such that $u\wedge v=0$ and $u\vee x=v\vee y=1$;
\item {\em conormal} iff it is dually normal, that is, for all $x,y\in L$, if $x\wedge y=0$, then there exist $u,v\in L$ such that $u\vee v=1$ and $u\wedge x=v\wedge y=0$.\end{itemize}

\begin{note}
In \cite{whcor} and \cite{paw}, the denominations above are reversed. We have adopted the version of these definitions in \cite{joh}, for the reason discussed by the author on \cite[p. $78$]{joh}.\label{notaden}\end{note}

\begin{proposition}
The following are equivalent:

\begin{enumerate}
\item\label{gelfand1} the bounded distributive lattice ${\rm Filt}(A)$ is normal;
\item\label{gelfand2} the bounded distributive lattice ${\rm PFilt}(A)$ is normal, that is: for all $x,y\in A$, if $[x)\vee [y)=A$, then there exist $u,v\in A$ such that $[u)\cap [v)=\{1\}$ and $[u)\vee [x)=[v)\vee [y)=A$;
\item\label{gelfand3} the bounded distributive lattice ${\cal L}(A)$ is conormal;
\item\label{gelfand4} any prime filter of $A$ is included in a unique maximal filter of $A$;
\item\label{gelfand5} any prime filter of ${\cal L}(A)$ is included in a unique maximal filter of ${\cal L}(A)$;
\item\label{gelfand6} the inclusion ${\rm Max}(A)\subseteq {\rm Spec}(A)$ admits a continuous retract;
\item\label{gelfand7} the inclusion ${\rm Max}({\cal L}(A))\subseteq {\rm Spec}({\cal L}(A))$ admits a continuous retract;
\item\label{gelfand8} ${\rm Spec}(A)$ is a normal topological space;
\item\label{gelfand9} ${\rm Spec}({\cal L}(A))$ is a normal topological space;
\item\label{gelfand10} for any $M\in {\rm Max}(A)$, $\{P\in {\rm Spec}(A)\ |\ P\subseteq M\}$ is a closed subset of ${\rm Spec}(A)$;
\item\label{gelfand11} for any $M\in {\rm Max}({\cal L}(A))$, $\{P\in {\rm Spec}({\cal L}(A))\ |\ P\subseteq M\}$ is a closed subset of ${\rm Spec}({\cal L}(A))$;
\item\label{gelfand12} for any $M\in {\rm Max}(A)$, $M$ is the only maximal filter of $A$ which includes $\displaystyle \bigcap _{\stackrel{\scriptstyle P\in {\rm Spec}(A)}{\scriptstyle P\subseteq M}}P$;
\item\label{gelfand13} for any $M\in {\rm Max}({\cal L}(A))$, $M$ is the only maximal filter of ${\cal L}(A)$ which includes $\displaystyle \bigcap _{\stackrel{\scriptstyle P\in {\rm Spec}({\cal L}(A))}{\scriptstyle P\subseteq M}}P$;
\item\label{gelfand14} any two distinct maximal filters of $A$ can be separated by disjoint open sets in ${\rm Spec}(A)$;
\item\label{gelfand15} any two distinct maximal filters of ${\cal L}(A)$ can be separated by disjoint open sets in ${\rm Spec}({\cal L}(A))$.
\end{enumerate}\label{gelfand}\end{proposition}

\begin{proof} (\ref{gelfand1})$\Rightarrow $(\ref{gelfand2}): Let $x,y\in A$, such that $[x)\vee [y)=A$. Then there exist $F,G\in {\rm Filt}(A)$ such that $F\cap G=[1)$ and $F\vee [x)=G\vee [y)=A$, thus there exist $u\in F$, $v\in G$ and $n\in \N ^*$ such that $u\odot x^n=v\odot y^n=0$. Hence $u\vee v\in F\cap G=[1)$, that is $u\vee v=1$, hence $[u)\cap [v)=[u\vee v)=[1)$, and $[u)\vee [x)=[u)\vee [x^n)=[u\odot x^n)=[0)=A=[0)=[v\odot y^n)=[v)\vee [y^n)=[v)\vee [y)$.

\noindent (\ref{gelfand2})$\Rightarrow $(\ref{gelfand1}): Let $F,G\in {\rm Filt}(A)$, such that $F\vee G=A$, that is there exist $x\in F$ and $y\in G$ such that $x\odot y=0$, thus $[x)\vee [y)=[x\odot y)=[0)=A$. Then there exist $u,v\in A$ such that $[u)\cap [v)=[1)$ and $[u)\vee [x)=[v)\vee [y)=A$. But $[x)\subseteq F$ and $[y)\subseteq G$, hence $[u)\vee F=[v)\vee G=A$.

\noindent (\ref{gelfand2})$\Leftrightarrow $(\ref{gelfand3}): By the fact that the lattice ${\cal L}(A)$ is isomorphic to the dual of ${\rm PFilt}(A)$, according to Proposition \ref{exsiunic} and Remark \ref{constrretic}.

\noindent \ref{gelfand3})$\Leftrightarrow $(\ref{gelfand5})$\Leftrightarrow $(\ref{gelfand7})$\Leftrightarrow $(\ref{gelfand9})$\Leftrightarrow $(\ref{gelfand11})$\Leftrightarrow $(\ref{gelfand13})$\Leftrightarrow $(\ref{gelfand15}): This is exactly \cite[Theorem $4$]{paw} applied to ${\cal L}(A)$.

\noindent (\ref{gelfand4})$\Leftrightarrow $(\ref{gelfand5}): By Proposition \ref{retic}, (\ref{retic1}), (\ref{retic2}) and (\ref{retic3}).

\noindent (\ref{gelfand6})$\Leftrightarrow $(\ref{gelfand7}): By Proposition \ref{retic}, (\ref{retic2}) and (\ref{retic3}).

\noindent (\ref{gelfand8})$\Leftrightarrow $(\ref{gelfand9}): By Proposition \ref{retic}, (\ref{retic2}).

\noindent (\ref{gelfand10})$\Leftrightarrow $(\ref{gelfand11}): By Proposition \ref{retic}, (\ref{retic1}), (\ref{retic2}) and (\ref{retic3}).

\noindent (\ref{gelfand12})$\Leftrightarrow $(\ref{gelfand13}): By Proposition \ref{retic}, (\ref{retic1}), (\ref{retic2}), (\ref{retic3}), and Remark \ref{intersarb}.

\noindent (\ref{gelfand14})$\Leftrightarrow $(\ref{gelfand15}): By Proposition \ref{retic}, (\ref{retic2}) and (\ref{retic3}).\end{proof}

Since the notion of a normal residuated lattice has another significance, we have chosen a denomination from ring theory for this particular kind of residuated lattices:

\begin{definition}
We shall call {\em residuated lattices with the Gelfand property} (or, in brief, {\em Gelfand residuated lattices}) those residuated lattices which have the equivalent properties from Proposition \ref{gelfand}.\end{definition}

Concerning condition (\ref{gelfand6}) from Proposition \ref{gelfand}:

\begin{proposition}
Let $A$ be a Gelfand residuated lattice. For every prime filter $P$ of $A$, let us denote by $M_P$ the only maximal filter of $A$ which includes $P$ (see condition (\ref{gelfand4}) from the same proposition), and define $\rho :{\rm Spec}(A)\rightarrow {\rm Max}(A)$ by $\rho (P)=M_P$ for every $P\in {\rm Spec}(A)$. Then $\rho $ is a continuous retract of the inclusion ${\rm Spec}(A)\subseteq {\rm Max}(A)$.\label{retract}\end{proposition}

\begin{proof} In the proof of \cite[Theorem $4$]{paw}, the author shows that, for any conormal bounded distributive lattice $L$, the function that takes every prime filter of $L$ to the only maximal filter which includes it is a continuous retract of the inclusion ${\rm Spec}(L)\subseteq {\rm Max}(L)$. The statement in the enunciation follows by applying this to ${\cal L}(A)$ instead of $L$ and using Proposition \ref{retic}, (\ref{retic1}), (\ref{retic2}) and (\ref{retic3}).\end{proof}

\begin{lemma}
If $F$ and $G$ are filters of $A$, then the following assertions are equivalent:

\begin{enumerate}
\item\label{l5.11} $F\cap G=\{1\}$ and $F\vee G=A$;
\item\label{l5.12} there exists $e\in {\cal B}(A)$ such that $F=[e)$ and $G=[\neg \, e)$.\end{enumerate}\label{l5.1}\end{lemma}

\begin{proof} (\ref{l5.11})$\Rightarrow $(\ref{l5.12}): If $F\vee G=A$, then there exist $e\in F$ and $f\in G$ such that $e\odot f=0$, that is $f\leq \neg \, e$, according to Lemma \ref{latrez}, (\ref{latrez1}), hence $\neg \, e\in G$, thus $[e)\subseteq F$ and $[\neg \, e)\subseteq G$. Then $e\vee f\in F\cap G=\{1\}$, which means that $e\vee f=1$, hence $e\vee \neg \, e=1$ by the inequality above. As shown by Lemma \ref{cbool}, (\ref{cbool1}) and (\ref{cbool0}), this means that $e\in {\cal B}(A)$, hence $\neg \, e\in {\cal B}(A)$ also. Now let $a\in F$ and $b\in G$. Then, according to Lemma \ref{cbool}, (\ref{cbool2}) and (\ref{cbool0}), and Lemma \ref{latrez}, (\ref{latrez5}), $e\rightarrow a=\neg \, e\vee a\in F\cap G=\{1\}$ and $\neg \, e\rightarrow b=\neg \, \neg \, e\vee b=e\vee b\in F\cap G=\{1\}$, so $e\rightarrow a=\neg \, e\rightarrow b=1$, that is $e\leq a$ and $\neg \, e\leq b$, thus $a\in [e)$ and $b\in [\neg \, e)$, hence $F\subseteq [e)$ and $G\subseteq [\neg \, e)$. Therefore $F=[e)$ and $G=[\neg \, e)$.

\noindent (\ref{l5.12})$\Rightarrow $(\ref{l5.11}): By Lemma \ref{cbool}, (\ref{cbool2}), and Lemma \ref{latrez}, (\ref{latrez1}).\end{proof}

\begin{proposition}
\begin{enumerate}
\item\label{clp1} ${\rm Clp}({\rm Spec}(A))=\{V(e)\ |\ e\in {\cal B}(A)\}=\{D(e)\ |\ e\in {\cal B}(A)\}$.
\item\label{clp2} If $A$ is Gelfand, then ${\rm Clp}({\rm Max}(A))=\{v(e)\ |\ e\in {\cal B}(A)\}=\{d(e)\ |\ e\in {\cal B}(A)\}$.
\item\label{clp3} If $A$ is semisimple, then ${\rm Clp}({\rm Max}(A))=\{v(e)\ |\ e\in {\cal B}(A)\}=\{d(e)\ |\ e\in {\cal B}(A)\}$.\end{enumerate}\label{clp}\end{proposition}

\begin{proof} (\ref{clp1}) The fact that $\{V(e)\ |\ e\in {\cal B}(A)\}=\{D(e)\ |\ e\in {\cal B}(A)\}\subseteq {\rm Clp}({\rm Spec}(A))$ is part of Corollary \ref{corvdb}.

For the converse inclusion, let us consider a clopen set $C$ of ${\rm Spec}(A)$. Then $C$ is a closed set of ${\rm Spec}(A)$, thus $C=V(F)$ for some $F\in {\rm Filt}(A)$. Also, it follows that $C$ is an open set of ${\rm Spec}(A)$, thus ${\rm Spec}(A)\setminus C={\rm Spec}(A)\setminus V(F)$ is a closed set of ${\rm Spec}(A)$, that is there exists $G\in {\rm Filt}(A)$ such that ${\rm Spec}(A)\setminus V(F)=V(G)$, that is ${\rm Spec}(A)=V(F)\cup V(G)=V(F\cap G)$ and $\emptyset =V(F)\cap V(G)=V(F\vee G)$, according to Proposition \ref{vsid}. Therefore $F\cap G=\{1\}$ and $A=F\vee G=\{x\in A\ |\ (\exists \, f\in F)\, (\exists \, g\in G)\, (f\odot g\leq x)\}$, hence there exist $f\in F,g\in G$ such that $f\odot g=0$, so $g\leq \neg \, f$ by Proposition \ref{latrez}, (\ref{latrez1}), thus $\neg \, f\in G$. So $f\in F$ and $\neg \, f\in G$, therefore $f\vee \neg \, f\in F\cap G=\{1\}$, thus $f\vee \neg \, f=1$, which means that $f\in {\cal B}(A)$ according to Lemma \ref{cbool}, (\ref{cbool1}).

Now let us prove that $V(F)=V(f)$. Since $f\in F$, we have $V(F)\subseteq V(f)$. Let $P\in V(f)$, so $f\in P$, thus $\neg \, f\notin P$, as shown by Proposition \ref{latrez}, (\ref{latrez1}), and the fact that $P$ is a prime and thus a proper filter of $A$. But $\neg \, f\in G$, hence $G\nsubseteq P$, that is $P\notin V(G)={\rm Spec}(A)\setminus V(F)$, therefore $P\in V(F)$. So $V(f)\subseteq V(F)$, thus $C=V(F)=V(f)$.

\noindent (\ref{clp2}) The fact that $\{v(e)\ |\ e\in {\cal B}(A)\}=\{d(e)\ |\ e\in {\cal B}(A)\}\subseteq {\rm Clp}({\rm Max}(A))$ is part of Corollary \ref{corvdb}.

For the converse inclusion, let $K$ be a clopen subset of ${\rm Max}(A)$ and denote $L=\{P\in {\rm Spec}(A)\ |\ (\exists \, M\in K)\, (P\subseteq M)\}$. Let $\rho :{\rm Spec}(A)\rightarrow {\rm Max}(A)$ be the continuous retract of the inclusion ${\rm Max}(A)\subseteq {\rm Spec}(A)$ from Proposition \ref{retract}. Then the inverse image $\rho ^{-1}(K)=L$, hence $L$ is a clopen subset of ${\rm Spec}(A)$ since $\rho $ is continuous, thus ${\rm Spec}(A)\setminus L$ is also a clopen subset of ${\rm Spec}(A)$. By (\ref{clp1}), this means that $L=D(e)$ and ${\rm Spec}(A)\setminus L=D(f)$ for some $e,f\in {\cal B}(A)$. Hence ${\rm Spec}(A)\setminus L={\rm Spec}(A)\setminus D(e)=V(e)$ and $L={\rm Spec}(A)\setminus D(f)=V(f)$, so every $P\in L$ contains $f$ and every $P\in {\rm Spec}(A)\setminus L$ contains $e$. Now let us denote $\displaystyle F=\bigcap _{P\in L}P$ and $\displaystyle G=\bigcap _{P\in {\rm Spec}(A)\setminus L}P$. Then $\displaystyle F\cap G=\bigcap _{P\in {\rm Spec}(A)}P=\{1\}$, $f\in F$ and $e\in G$, so $e\vee f\in F\vee G$. Assume by absurdum that $F\vee G\neq A$, which means that there exists $P\in {\rm Spec}(A)$ such that $F\vee G\subseteq P$, thus $e\vee f\in P$. Since ${\rm Spec}(A)=L\cup ({\rm Spec}(A)\setminus L)=D(e)\cup D(f)$, it follows that $P\in D(e)$ or $P\in D(f)$, therefore $e\notin P$ or $f\notin P$. We have obtaines a contradiction with the primality of $P$. We conclude that $F\vee G=A$. According to Lemma \ref{l5.1}, we get that there exists $g\in {\cal B}(A)$ such that $F=[g)$ and $G=[\neg \, g)$, that is $\displaystyle \bigcap _{P\in L}P=[g)$ and $\displaystyle \bigcap _{P\in {\rm Spec}(A)\setminus L}P=[\neg \, g)$. These equalities show that the following hold for every $Q\in {\rm Spec}(A)$: if $Q\in L$, then $g\in Q$, thus $Q\in V(g)$, while, if $Q\in {\rm Spec}(A)\setminus L$, then $\neg \, g\in Q$, thus $Q\in V(\neg \, g)$. Therefore $L\subseteq V(g)$ and ${\rm Spec}(A)\setminus L\subseteq V(\neg \, g)={\rm Spec}(A)\setminus V(g)$ by Lemma \ref{vsidbool}, so $V(g)\subseteq L$, hence $L=V(g)$, thus we have, for the image of $L$ through $\rho $: $K=\rho (L)=\rho (V(g))=v(g)$, where the last equality follows immediately by double inclusion.

\noindent (\ref{clp3}) The fact that $\{v(e)\ |\ e\in {\cal B}(A)\}=\{d(e)\ |\ e\in {\cal B}(A)\}\subseteq {\rm Clp}({\rm Max}(A))$ is part of Corollary \ref{corvdb}. Below we prove the converse inclusion.

Let $K$ be a clopen subset of ${\rm Max}(A)$, which is a compact space, so that $K$ is compact. Since $\{d(a)\ |\ a\in A\}$ is a basis of open sets for ${\rm Max}(A)$, it follows that $\displaystyle K=\bigcup _{i\in I}d(a_i)$ for some $(a_i)_{i\in I}\subseteq A$. But $K$ is compact, thus there exists a finite subset $I_0$ of $I$ such that $\displaystyle K=\bigcup _{i\in I_0}d(a_i)=d(a)$, with $\displaystyle a=\bigwedge _{i\in I_0}a_i$, as shown by Proposition \ref{vsidelem}, (\ref{vsidelem1}). Thus ${\rm Clp}({\rm Max}(A))\subseteq \{d(x)\ |\ x\in A\}$. $d(a)$ is clopen, thus $v(a)$ is clopen in ${\rm Max}(A)$, so $v(a)=d(b)$ for some $b\in A$. But, again by by Proposition \ref{vsidelem}, (\ref{vsidelem1}), $d(a\odot b)=d(a)\vee d(b)=d(a)\vee v(a)\supseteq d(a)\cup v(a)={\rm Max}(A)$, so $d(a\odot b)={\rm Max}(A)$, which means that $a^n\odot b^n=(a\odot b)^n=0$ for some $n\in \N ^*$. According to Proposition \ref{vsidelem}, (\ref{vsidelem2}), $d(a^n\vee b^n)=d(a^n)\cap d(b^n)=d(a)\cap d
(b)=d(a)\cap v(a)=\emptyset $, hence $a^n\vee b^n\in {\rm Rad}(A)=\{1\}$, because $A$ is semisimple. By Lemma \ref{latrez}, (\ref{latrez10}), it follows that $a^n\vee b^n=1$ and $a^n\wedge b^n=0$, hence $a^n,b^n\in {\cal B}(A)$. Thus $K=d(a)=d(a^n)$, with $a^n\in {\cal B}(A)$.\end{proof}

In the following, we shall make repeated use of the definition of Gelfand residuated lattices through condition (\ref{gelfand4}) from Proposition \ref{gelfand}, to which we have appealled in the Proposition \ref{retract} above as well. We shall not recall this condition in the results below.

\begin{example}
Any local residuated lattice is Gelfand. Since any linearly orderred residuated lattice is local (see Remark \ref{celeloc}), we get that any linearly orderred residuated lattice is Gelfand.\label{locsuntg}\end{example}

\begin{example}
It is known (\cite{kow}, \cite{pic}) that $A$ is a hyperarchimedean residuated lattice iff ${\rm Spec}(A)={\rm Max}(A)$. Consequently, any hyperarchimedean residuated lattice has the Gelfand property. In particular, any Boolean algebra is a Gelfand residuated lattice.\end{example}

\begin{example}
Any MV--algebra is a Gelfand residuated lattice. Moreover, any BL--algebra is a Gelfand residuated lattice. These facts follow from Remark \ref{blgelfand}.\label{exmvbl}\end{example}

\begin{example}
The residuated lattice $A$ in Example \ref{exlpdif} is not Gelfand. Indeed, $\{1\}\in {\rm Spec}(A)$ and ${\rm Max}(A)=\{[a),[b)\}$, with $[a)\neq [b)$.\label{nicig}\end{example}

In \cite{ggcm}, we have introduced these properties for a residuated lattice $A$:

\begin{center}
\begin{tabular}{rl}
$(\star )$ & for all $x\in A$, there exist $u\in {\rm Rad}(A)$ and $e\in {\cal B}(A)$ such that $[x)=[u)\vee [e)$;\\
$(\star \star )$ & for all $x\in A$, there exist $u\in A$ and $e\in {\cal B}(A)$ such that $\neg \, u$ is nilpotent and $[x)=[u)\vee [e)$.\end{tabular}
\end{center}

\noindent and we have proved that $(\star )\Rightarrow $ BLP $\Rightarrow (\star \star )$.

\begin{remark}
Example \ref{nicig} suggests a method for obtaining residuated lattices without the Gelfand property, by taking into account the obvious fact that a residuated lattice whose trivial filter is prime is Gelfand iff it is local: take a residuated lattice $R$ which is not local, and let $A$ be the ordinal sum between $R$ and a non--trivial residuated chain $C$ (see \cite{gal}, \cite[p. 129, 136, 137]{ior} for this particular construction). Then $\{1\}\in {\rm Spec}(A)$, while ${\rm Max}(A)=\{M\cup C\ |\ M\in {\rm Max}(R)\}$, so $A$ is not local either, hence $A$ is not Gelfand.

Actually, as shown by \cite[Remark $6.32$]{ggcm}, if $A$ is the ordinal sum between a residuated lattice $L$ and a non--trivial chain $C$, then: $A$ is Gelfand iff $A$ is local iff $L$ is local iff $A$ has BLP iff $A$ has $(\star )$ iff $A$ has $(\star \star )$.\end{remark}

\begin{lemma}{\rm \cite{coloqu}} If $L$ is a bounded distributive lattice, then ${\rm Rad}(L)=\{a\in L\ |\ (\forall \, x\in L)\, (a\wedge x=0\Rightarrow x=0)\}$.\label{radd01}\end{lemma}

\begin{proposition} The radical of any conormal bounded distributive lattice has BLP.\label{radco}\end{proposition}

\begin{proof} Let $L$ be a conormal bounded distributive lattice, and let us denote by $R={\rm Rad}(L)$. Take an arbitrary $x\in L$ such that $x/R\in {\cal B}(L/R)$, so that there exists an $y\in L$ with $(x\vee y)/R=x/R\vee y/R=1/R$ and $(x\wedge y)/R=x/R\wedge y/R=0/R$. $(x\wedge y)/R=0/R$ means that there exists $a\in R$ such that $a\wedge x\wedge y=a\wedge 0=0$. From Lemma \ref{radd01} and the conormality of $L$, we get that $x\wedge y=0$, hence there exist $u,v\in L$ such that $u\vee v=1$ and $u\wedge x=v\wedge y=0$, therefore $(x\vee y)\wedge u\wedge v=(x\wedge u\wedge v)\vee (y\wedge u\wedge v)=(0\wedge v)\vee (0\wedge u)=0\vee 0=0$. Also, that fact that $(x\vee y)/R=1/R$ means that $x\vee y\in R$, and, again by Lemma \ref{radd01}, we get that $u\wedge v=0$. We have obtained $u\vee v=1$ and $u\wedge v=0$, thus $u,v\in {\cal B}(L)$. But $u\wedge x=0$, hence $v=v\vee 0=v\vee (u\wedge x)=(v\vee u)\wedge (v\vee x)=1\wedge (v\vee x)=v\vee x$, thus $v=v\vee x$, therefore $x\leq v$, hence $x/R\leq v/R$. Now we apply the fact that $x/R\vee y/R=1/R$ and we deduce the following: $v/R=v/R\wedge 1/R=v/R\wedge (x/R\vee y/R)=(v/R\wedge x/R)\vee (v/R\wedge y/R)=(v/R\wedge x/R)\vee (v\wedge y)/R=(v/R\wedge x/R)\vee 0/R=v/R\wedge x/R$, so $v/R=v/R\wedge x/R$, hence $v/R\leq x/R$. Therefore $x/R=v/R\in {\cal B}(L)/R$. This means that $R={\rm Rad}(L)$ has BLP in $L$.\end{proof}

\begin{corollary} \begin{enumerate}
\item\label{radgelfand1} The radical of any Gelfand residuated lattice has BLP.
\item\label{radgelfand2} If $A$ is an MV--algebra or a BL--algebra, then ${\rm Rad}(A)$ has BLP.\end{enumerate}\label{radgelfand}\end{corollary}

\begin{proof} (\ref{radgelfand1}) By condition (\ref{gelfand3}) from Proposition (\ref{gelfand}), Proposition \ref{radco} applied to the bounded distributive lattice ${\cal L}(A)$, Corollary \ref{transferad}, (\ref{transferad3}), and Proposition \ref{blpretic}.

\noindent (\ref{radgelfand2}) follows from (\ref{radgelfand1}) and Example \ref{exmvbl}.\end{proof}

(\ref{radgelfand2}) from the previous corollary is actually a well--known result, but this corollary shows an interesting way of deriving this result from a more general context.

\begin{proposition}{\rm \cite{ggcm}}
\begin{enumerate}
\item\label{dinartblp1} The following are equivalent:
\begin{itemize}
\item $A$ is semilocal and ${\rm Rad}(A)$ has BLP;
\item $A$ is semilocal and has BLP;
\item $A$ is semilocal and satisfies $(\star )$;
\item $A$ is isomorphic to a finite direct product of local residuated lattices.
\end{itemize}
\item\label{dinartblp2} The following are equivalent:
\begin{itemize}
\item $A$ is maximal and ${\rm Rad}(A)$ has BLP;
\item $A$ is maximal and has BLP;
\item $A$ is maximal and satisfies $(\star )$;
\item $A$ is isomorphic to a finite direct product of local maximal residuated lattices.\end{itemize}\end{enumerate}\label{dinartblp}\end{proposition}

\begin{corollary}
\begin{itemize}
\item Semilocal Gelfand residuated lattices have BLP. In particular, semilocal MV--algebras, semilocal BL--algebras and maximal Gelfand residuated lattices have BLP.
\item Moreover, semilocal Gelfand residuated lattices satisfy $(\star )$. In particular, semilocal MV--algebras, semilocal BL--algebras and maximal Gelfand residuated lattices satisfy $(\star )$.
\item Every semilocal Gelfand residuated lattice is isomorphic to a finite direct product of local residuated lattices.
\item Every maximal Gelfand residuated lattice is isomorphic to a finite direct product of local maximal residuated lattices.
\end{itemize}\label{inartblptop}\end{corollary}

\begin{proof} By Corollary \ref{radgelfand} and Proposition \ref{dinartblp}.\end{proof}

\begin{theorem}
The following are equivalent:

\begin{enumerate}
\item\label{p5.21} $A$ has BLP;
\item\label{p5.22} ${\rm Spec}(A)$ is strongly zero--dimensional.\end{enumerate}\label{p5.2}\end{theorem}

\begin{proof} (\ref{p5.21})$\Rightarrow $(\ref{p5.22}): We will verify the condition (\ref{multcit3}) from Proposition \ref{multcit}. Let $U$ and $V$ be open sets in ${\rm Spec}(A)$ such that ${\rm Spec}(A)=U\cup V$. Then $U=D(F)$ and $V=D(G)$ for some filters $F$, $G$ of $A$. Thus $D(F\vee G)=D(F)\cup D(G)={\rm Spec}(A)$ by Proposition \ref{vsid}, (\ref{vsid1}), so $F\vee G= A$, hence there exist $x\in F$ and $y\in G$ such that $x\odot y=0$, hence $y\leq \neg \, x$ by Lemma \ref{latrez}, (\ref{latrez1}). Thus $\neg \, x\in G$. The fact that $A$ has BLP and Proposition \ref{propblp} imply that there exists $e\in {\cal B}(A)$ such that $e\in [x)\subseteq F$ and $\neg \, e\in [\neg \, x)\subseteq G$, therefore $D(e)\subseteq D(F)$ and $D(\neg \, e)\subseteq D(G)$. According to Corollary \ref{corvdb}, $D(e)$ and $D(\neg \, e)$ are clopen subsets of ${\rm Spec}(A)$. By Lemma \ref{latrez}, (\ref{latrez1}), Lemma \ref{cbool}, (\ref{cbool1}), and Proposition \ref{vsidelem}, we have $D(e)\cup D(\neg \, e)=D(e\odot \neg \, e)=D(0)={\rm Spec}(A)$ and $D(e)\cap D(\neg \, e)=D(e\vee \neg \, e)=D(1)=\emptyset $. 

\noindent (\ref{p5.22})$\Rightarrow $(\ref{p5.21}): Let $a\in A$. By Lemma \ref{latrez}, (\ref{latrez1}), and Proposition \ref{vsidelem}, $D(a)\cup D(\neg \, a)=D(a\odot \neg \, a)=D(0)={\rm Spec}(A)$, hence, by condition (\ref{multcit3}) from Proposition \ref{multcit}, there exist two clopen sets $C$ and $D$ in ${\rm Spec}(A)$ such that $C\subseteq D(a)$, $D\subseteq D(\neg \, a)$, $C\cap D=\emptyset $ and $C\cup D={\rm Spec}(A)$. Proposition \ref{clp} ensures us that $C=D(e)$ and $D=D(f)$ for some $e,f\in {\cal B}(A)$. Then ${\rm Spec}(A)=D(e)\cup D(f)=D(e\odot f)$ by Proposition \ref{vsidelem}, hence $e\odot f=0$, thus $f\leq \neg \, e$ by Lemma \ref{latrez}, (\ref{latrez1}), and so $D(e)=C\subseteq D(a)$ and $D(\neg \, e)\subseteq D(f)=D\subseteq D(\neg \, a)$. By Lemma \ref{incld}, it follows that $[e)\subseteq [a)$ and $[\neg \, e)\subseteq [\neg \, a)$, hence $e\in [a)$ and $\neg \, e\in [\neg \, a)$. Therefore $A$ has BLP by Proposition \ref{propblp}.\end{proof}

\begin{theorem} The following assertions are equivalent:

\begin{enumerate}
\item\label{p5.31} $A$ has BLP;
\item\label{p5.32} for any $M,N\in {\rm Max}(A)$ such that $M\neq N$, there exists an $e\in {\cal B}(A)$ such that $e\in M$ and $\neg \, e\in N$;
\item\label{p5.33} the family ${\cal E}=\{d(e)\ |\ e\in {\cal B}(A)\}=\{v(e)\ |\ e\in {\cal B}(A)\}$ is a basis of ${\rm Max}(A)$;
\item\label{p5.310} for any $a\in A$, there exists an $e\in {\cal B}(A)$ such that $v(a)\subseteq d(e)$ and $v(\neg \, a)\subseteq v(e)$;
\item\label{p5.35} $A$ is Gelfand and ${\rm Max}(A)$ is zero--dimensional;
\item\label{p5.36} $A$ is Gelfand and ${\rm Max}(A)$ is strongly zero--dimensional;
\item\label{p5.37} $A$ is Gelfand and ${\rm Max}(A)$ is normal;
\item\label{p5.38} $A$ is Gelfand and ${\rm Max}(A)$ is Boolean.
\end{enumerate}\label{p5.3}\end{theorem}

\begin{proof} (\ref{p5.31})$\Rightarrow $(\ref{p5.32}): Assume that $A$ has BLP and let $M$, $N$ be distinct maximal filters of $A$. Then $M\vee N=A$, so there exist $a\in M$ and $b\in N$ such that $a\odot b=0$, hence $b\leq \neg \, a$ by Lemma \ref{latrez}, (\ref{latrez1}). By Proposition \ref{propblp}, there exists $e\in {\cal B}(A)$ such that $e\in [a)\subseteq M$ and $\neg \, e\in [\neg \, a)\subseteq [b)\subseteq N$, hence $e\in M$ and $\neg \, e\in N$.

\noindent (\ref{p5.32})$\Rightarrow $(\ref{p5.33}): The second equality in (\ref{p5.33}) comes from Lemma \ref{vsidbool}. Now let $U$ be an open subset of ${\rm Max}(A)$, so that $K={\rm Max}(A)\setminus U$ is closed, and let $M\in U$, arbitrary. Then the hypothesis (\ref{p5.32}) of this implication ensures us that, for every $N\in K$, there exists $e_N\in {\cal B}(A)$ such that $e_N\in N$ and $\neg \, e_N\in M$. So $N\in v(e_N)$ for all $N\in K$, hence $\{v(e_N)\ |\ N\in K\}$ is an open cover of $K$ in ${\rm Max}(A)$, according to Corollary \ref{corvdb}. $K$ is a closed subset of the compact space ${\rm Max}(A)$, so $K$ is also compact. Hence there exist $k\in \N ^*$ and $N_1,\ldots ,N_k\in K$ such that $\displaystyle K\subseteq \bigcup _{i=1}^kv(e_{\scriptstyle N_{\scriptstyle i}})=v(e)$, where, according to Proposition \ref{vsidelem}, (\ref{vsidelem2}), $\displaystyle e=\bigvee _{i=1}^ke_{\scriptstyle N_{\scriptstyle i}}\in {\cal B}(A)$. Therefore $K\subseteq {\rm Max}(A)\setminus d(e)$, hence $d(e)\subseteq {\rm Max}(A)\setminus K=U$. By Lemma \ref{cbool}, (\ref{cbool0}), and the fact that $\neg \, e_1,\ldots ,\neg \, e_k\in M$, $\displaystyle \neg \, e=\neg \, (\bigvee _{i=1}^ke_i)=\bigwedge _{i=1}^k\neg \, e_i\in M$, thus $M\in d(e)$. We have obtained that $M\in d(e)\subseteq U$. Therefore ${\cal E}$ is a basis for ${\rm Max}(A)$.

\noindent (\ref{p5.33})$\Rightarrow $(\ref{p5.310}): The hypothesis of this implication and Corollary \ref{corvdb} show that ${\rm Max}(A)$ has a basis of clopen subsets, so it is zero--dimensional, thus Boolean by Corollary \ref{coolmix}. Now let $K$ be a clopen subset of ${\rm Max}(A)$, hence, by condition (\ref{p5.33}), $\displaystyle K=\bigcup _{i\in I}d(e_i)$ for some non--empty family $(e_i)_{i\in I}$ of Boolean elements of $A$. ${\rm Max}(A)$ is compact and $K$ is closed in ${\rm Max}(A)$, hence $K$ is compact, thus there exists a finite non--empty subset $J$ of $I$ such that $\displaystyle K=\bigcup _{i\in J}d(e_i)=d(e)$, where $\displaystyle e=\bigwedge _{i\in J}e_i\in {\cal B}(A)$. We have applied Proposition \ref{vsidelem}, (\ref{vsidelem1}). Hence $K=d(e)\in {\cal E}$, that is any clopen subset of ${\rm Max}(A)$ belongs to the family ${\cal E}$ of clopen subsets of ${\rm Max}(A)$, which means that ${\rm Clp}({\rm Max}(A))={\cal E}$. Let $a\in A$. Using the fact that $d(a)=d([a))$, Proposition \ref{vsidelem}, (\ref{vsidelem1}) and Lemma \ref{latrez}, (\ref{latrez1}), we get that $d(a)\cup d(\neg \, a)=d(a)\cup d(\neg \, a)=d(a \odot \neg \, a)=d(0)={\rm Max}(A)$. But both $d(a)$ and $d(\neg \, a)$ are open subsets of ${\rm Max}(A)$, hence $d(a),d(\neg \, a)\in {\rm Clp}({\rm Max}(A))={\cal E}$, that is there exist $e,f\in {\cal B}(A)$ such that $d(a)=d(f)$ and $d(\neg \, a)=d(e)$, therefore $v(a)=v(f)$ and $v(\neg \, a)=v(e)$ and, again by Proposition \ref{vsidelem}, (\ref{vsidelem1}), and Lemma \ref{latrez}, (\ref{latrez1}), $d(f\odot e)=d(f)\cup d(e)=d(a)\cup d(\neg \, a)={\rm Max}(A)$, thus $f\odot e=0$, hence $f\leq \neg \, e$, and $d(e)\cup d(\neg \, e)=d(e\odot \neg \, e)=d(0)={\rm Max}(A)$, which means that $v(a)=v(f)\subseteq v(\neg \, e)={\rm Max}(A)\setminus d(\neg \, e)\subseteq d(e)$, so $v(a)\subseteq d(e)$.

\noindent (\ref{p5.310})$\Rightarrow $(\ref{p5.31}): Let $a\in A$, arbitrary. Let $e\in {\cal B}(A)$ such that $v(a)\subseteq d(e)$ and $v(\neg \, a)\subseteq v(e)=d(\neg \, e)$, according to Lemma \ref{vsidbool}. By Lemma \ref{ltextnou} Lemma \ref{cbool}, (\ref{cbool3}), and Lemma \ref{latrez}, (\ref{latrez1}), from the fact that $v(a)\subseteq d(e)$ we get that $a\odot e$ is nilpotent, thus there exists an $n\in \N ^*$ such that $(a\odot e)^n=a^n\odot e^n=a^n\odot e=0$, hence $a^n\leq \neg \, e$, so $\neg \, e\in [a)$. Analogously, from the fact that $v(\neg \, a)\subseteq d(\neg \, e)$ we get that $e=\neg \, \neg \, e\in [\neg \, a)$, by Lemma \ref{cbool}, (\ref{cbool0}). According to Proposition \ref{propblp}, this means that $A$ has BLP.

\noindent (\ref{p5.33})$\Rightarrow $(\ref{p5.35}): By Corollary \ref{corvdb}.

\noindent (\ref{p5.35})$\Rightarrow $(\ref{p5.33}): By Proposition \ref{clp}.

\noindent (\ref{p5.35})$\Leftrightarrow $(\ref{p5.36})$\Leftrightarrow $(\ref{p5.37})$\Leftrightarrow $(\ref{p5.38}): By Corollary \ref{coolmix}.\end{proof}

\begin{remark}
Since $(\star )$ implies BLP, it follows that $(\star )$ implies each of the conditions in the previous theorem. As shown by Proposition \ref{dinartblp}, if $A$ is semilocal, then, furthermore, each of the conditions in Theorem \ref{p5.3} is equivalent to $(\star )$. Next we give a topological characterization for $(\star )$, too.\end{remark}

\begin{theorem}
The following are equivalent:

\begin{enumerate}
\item\label{charstar0} $A$ satisfies $(\star )$;
\item\label{charstar2} for any $a\in A$, there exists an $e\in {\cal B}(A)$ such that $a\odot e$ is nilpotent and $a\vee e\in {\rm Rad}(A)$;
\item\label{charstar3} for any $a\in A$, there exists an $e\in {\cal B}(A)$ such that $v(a)\subseteq d(e)$ and $d(a)\subseteq v(e)$;
\item\label{charstar1} for any $a\in A$, there exists an $e\in {\cal B}(A)$ such that $v(a)\subseteq d(e)$ and, for all $n\in \N ^*$, $v(\neg \, a^n)\subseteq v(e)$.\end{enumerate}
\label{charstar}\end{theorem}

\begin{proof} We shall make repeated use of Proposition \ref{cbool}, (\ref{cbool3}) and (\ref{cbool0}), and Proposition \ref{latrez}, (\ref{latrez1}) and (\ref{latrez6}).

\noindent (\ref{charstar0})$\Rightarrow $(\ref{charstar2}): Let $a\in A$. Since $A$ satisfies $(\star )$, it follows that there exist $u\in {\rm Rad}(A)$ and $f\in {\cal B}(A)$ such that $[a)=[u)\vee [f)=[u\odot f)$, thus $a^n\leq u\odot f$ and $(u\odot f)^n=u^n\odot f\leq a$ for some $n\in \N ^*$. Denote $e=\neg \, f\in {\cal B}(A)$. Then $(a\odot e)^n=a^n\odot e\leq u\odot f\odot e=u\odot f\odot \neg \, f=u\odot 0=0$, so $a\odot e$ is nilpotent. By the law of residuation and Proposition \ref{cbool}, (\ref{cbool2}), $u^n\odot f\leq a$ is equivalent to $u^n\leq f\rightarrow a=\neg \, f\vee a=e\vee a$; since $u\in {\rm Rad}(A)$, we have $u^n\in {\rm Rad}(A)$, hence $e\vee a\in {\rm Rad}(A)$.

\noindent (\ref{charstar2})$\Rightarrow $(\ref{charstar0}): Let $a\in A$. Then there exists $e\in {\cal B}(A)$ such that $a\odot e$ is nilpotent and $a\vee e\in {\rm Rad}(A)$. Denote $u=a\vee e\in {\rm Rad}(A)$ and $f=\neg \, e\in {\cal B}(A)$. Then, by Proposition \ref{latrez}, (\ref{latrez9}), $u\odot \neg \, e=(a\vee e)\odot \neg \, e=(a\odot \neg \, e)\vee (e\odot \neg \, e)=(a\odot \neg \, e)\vee 0=(a\odot \neg \, e)$, thus $u\odot f=a\odot f\leq a$, so $u\odot f\leq a$. The fact that $a\odot e$ is nilpotent shows that $(a\odot e^n)=a^n\odot e^n=a^n\odot e=0$ for some $n\in \N ^*$, hence $a^n\leq \neg \, e=f$, so $a^{n+1}\leq a\odot f=u\odot f\leq a$, hence $[a)=[a^{n+1})\supseteq [u\odot f)\supseteq [a)$, thus $[a)=[u\odot f)=[u)\vee [f)$. Therefore $A$ satisfies $(\star )$.

\noindent (\ref{charstar2})$\Leftrightarrow $(\ref{charstar3}):
 By Lemma \ref{ltextnou}.

\noindent (\ref{charstar3})$\Leftrightarrow $(\ref{charstar1}):
 By Lemma \ref{l4textnou}.\end{proof}

The statement on MV--algebras and BL--algebras in the next result follows from Corollary \ref{inartblptop}.

\begin{corollary} If $A$ is Gelfand, then the following are equivalent:

\begin{enumerate}
\item\label{clui5.31} $A$ has BLP;
\item\label{clui5.33} the family ${\cal E}=\{d(e)\ |\ e\in {\cal B}(A)\}=\{v(e)\ |\ e\in {\cal B}(A)\}$ is a basis of ${\rm Max}(A)$;
\item\label{clui5.35} ${\rm Max}(A)$ is zero--dimensional;
\item\label{clui5.36} ${\rm Max}(A)$ is strongly zero--dimensional;
\item\label{clui5.37} ${\rm Max}(A)$ is normal;
\item\label{clui5.38} ${\rm Max}(A)$ is Boolean.
\end{enumerate}

In particular, if $A$ is an MV--algebra or a BL--algebra, then the previous assertions are equivalent.\label{clui5.3}\end{corollary}

We recall that, if $(G,u)$ is an abelian unital lattice--orderred--group (in brief, $l$--group), in additive notation, and $G_+=\{a\in G\ |\ 0\leq a\}$, then an element $e\in G$ is called a {\em component} of $G$ iff there exists an $f\in G$ such that $e\wedge f=0$ and $e+f$ is a strong unit of $G$. By \cite[p. 984]{hkg}, $(G,u)$ is said to be {\em clean} iff, for every $a\in G$, there exists a strong unit $v\in G$ and a component $e\in G$ such that $a=v+e$. By ${\rm Max}(G)$ we denote the set of maximal proper convex $l$--subgroups of $G$ endowed with the hull--kernel topology; ${\rm Max}(G)$ is a compact Hausdorff space. If $A$ is an MV--algebra and $(G,u)$ is the abelian unital $l$--group associated to $A$ through the Mundici correspondence, then the topological spaces ${\rm Max}(A)$ and ${\rm Max}(G)$ are homeomorphic, according to \cite{cdom} and \cite{ile}. The following result establishes the relation between MV--algebras with BLP and clean $l$--groups. 

\begin{proposition}
Let $A$ be an MV--algebra and $(G,u)$ its associated unital $l$--group. Then the following are equivalent:

\begin{enumerate}
\item\label{marvel1} $A$ has BLP; 
\item\label{marvel8} ${\rm Max}(A)$ is Boolean;
\item\label{marvel8g} ${\rm Max}(G)$ is Boolean;
\item\label{marvel1g} $(G,u)$ is a clean unital $l$--group.\end{enumerate}\label{marvel}\end{proposition}

\begin{proof} (\ref{marvel1})$\Leftrightarrow $(\ref{marvel8}) is part of Corollary \ref{clui5.3}. The fact that ${\rm Max}(A)$ and ${\rm Max}(G)$ are homeomorphic proves that (\ref{marvel8})$\Leftrightarrow $(\ref{marvel8g}). (\ref{marvel1g})$\Leftrightarrow $(\ref{marvel8g}) is part of \cite[Theorem $3.5$]{hkg}.\end{proof}

\begin{corollary}
Residuated lattices with BLP form a subclass of the class of Gelfand residuated lattices.\end{corollary}

\begin{corollary}\begin{itemize}
\item If $A$ is Gelfand and semilocal, then ${\rm Max}(A)$ is Boolean, thus also strongly zero--dimensional.
\item If $A$ is a semilocal MV--algebra, then ${\rm Max}(A)$ is Boolean, thus also strongly zero--dimensional.
\item If $A$ is a semilocal BL--algebra, then ${\rm Max}(A)$ is Boolean, thus also strongly zero--dimensional.\end{itemize}\label{ulica}\end{corollary}

\begin{proof} By Corollaries \ref{inartblptop} and \ref{coolmix} and Theorem \ref{p5.3}.\end{proof}

\begin{proposition}
If $A$ is semisimple and the topological space ${\rm Max}(A)$ is Hausdorff, then $A$ is Gelfand.\label{propozitia3}\end{proposition}

\begin{proof} Assume that $A$ is semisimple and ${\rm Max}(A)$ is Hausdorff, and let $M$ and $N$ be distinct maximal filters of $A$. Then, according to Lemma \ref{lema1}, ${\rm Max}(A)$ is dense in ${\rm Spec}(A)$, thus the intersection between ${\rm Max}(A)$ and any non--empty open set of ${\rm Spec}(A)$ is non--empty. Also, there exist $a,b\in A$ such that $M\in d(a)$, $N\in d(b)$ and $d(a)\cap d(b)=\emptyset $, that is $D(a)\cap D(b)\cap {\rm Max}(A)=\emptyset $. Therefore $D(a)\cap D(b)=\emptyset $. Hence $M$ and $N$ are separated by the open sets $D(a)$ and $D(b)$ in ${\rm Spec}(A)$, which means that $A$ is Gelfand, as condition (\ref{gelfand14}) from Proposition \ref{gelfand} ensures us.\end{proof}

\begin{corollary} \begin{enumerate}
\item\label{corbool1} If $A$ is semisimple and ${\rm Max}(A)$ is Boolean, then $A$ has BLP.
\item\label{corbool2} ${\rm Max}(A)$ is Boolean iff $A/{\rm Rad}(A)$ has BLP.\end{enumerate}\label{corbool}\end{corollary}

\begin{proof} (\ref{corbool1}) Assume that $A$ is semisimple and the topological space ${\rm Max}(A)$ is Boolean, thus Hausdorff. By Proposition \ref{propozitia3} and Theorem \ref{p5.3}, it follows that $A$ is Gelfand, hence $A$ has BLP.

\noindent (\ref{corbool2}) $\Rightarrow $: By (\ref{corbool1}), the fact that $A/{\rm Rad}(A)$ is semisimple and Lemma \ref{lema2}.

\noindent $\Leftarrow $: If $A/{\rm Rad}(A)$ has BLP, then ${\rm Max}(A/{\rm Rad}(A))$ is Boolean by Theorem \ref{p5.3}, hence ${\rm Max}(A)$ is Boolean by Lemma \ref{lema2}.\end{proof}

\begin{remark} According to Corollary \ref{coolmix}, the statements in Corollary \ref{corbool} hold if we replace ``Boolean`` by either of these kinds of topological space: zero--dimensional, strongly zero--dimensional, normal.\end{remark}

\begin{corollary}
If $A$ is semisimple, then the following are equivalent:

\begin{itemize}
\item $A$ has BLP;
\item ${\rm Max}(A)$ is zero--dimensional;
\item ${\rm Max}(A)$ is strongly zero--dimensional;
\item ${\rm Max}(A)$ is normal;
\item ${\rm Max}(A)$ is Boolean.\end{itemize}\label{sweet}\end{corollary}

\begin{proof} By Theorem \ref{p5.3} and Corollary \ref{corbool}.\end{proof}

\begin{corollary}
If $A$ is semisimple and semilocal, then the following are equivalent:

\begin{itemize}
\item $A$ satisfies $(\star )$;
\item $A$ has BLP;
\item $A$ is Gelfand;
\item ${\rm Max}(A)$ is zero--dimensional;
\item ${\rm Max}(A)$ is strongly zero--dimensional;
\item ${\rm Max}(A)$ is normal;
\item ${\rm Max}(A)$ is Boolean.\end{itemize}\label{sweeter}\end{corollary}

\begin{proof} By Proposition \ref{dinartblp} and Corollaries \ref{inartblptop} and \ref{sweet}.\end{proof}

\end{document}